\documentclass[a4paper]{amsart}
\usepackage[alphabetic]{amsrefs}
\usepackage{amssymb,mathrsfs}
\usepackage[all]{xy}
\newdir{ >}{{}*!/-5pt/@{>}} 
\SelectTips{cm}{} 
\usepackage{amssymb}
\usepackage{array}
\usepackage{lscape}
\usepackage{tikz}

\title[The Adams spectral sequence for the image-of-$J$ spectrum]
	{The Adams spectral sequence \\ for the image-of-$J$ spectrum}
\author{Robert R. Bruner}
\address{Department of Mathematics, Wayne State University, USA}
\email{robert.bruner@wayne.edu}
\author{John Rognes}
\address{Department of Mathematics, University of Oslo, Norway}
\email{rognes@math.uio.no}
\subjclass[2010]{55Q50, 55T15}

\date{January 26th 2022}

\newtheorem{theorem}{Theorem}[section]
\newtheorem{proposition}[theorem]{Proposition}
\newtheorem{lemma}[theorem]{Lemma}
\theoremstyle{definition}
\newtheorem{definition}[theorem]{Definition}
\theoremstyle{remark}
\newtheorem{remark}[theorem]{Remark}

\numberwithin{equation}{section}

\DeclareMathOperator{\cof}{cof}
\DeclareMathOperator{\coim}{coim}
\DeclareMathOperator{\cok}{cok}
\DeclareMathOperator{\Ext}{Ext}
\DeclareMathOperator{\Hom}{Hom}
\DeclareMathOperator{\im}{im}
\DeclareMathOperator{\ord}{ord}
\newcommand{\bF}{\mathbb{F}}
\newcommand{\bZ}{\mathbb{Z}}
\newcommand{\from}{\leftarrow}
\newcommand{\into}{\rightarrowtail}
\newcommand{\longfrom}{\longleftarrow}
\newcommand{\longto}{\longrightarrow}
\newcommand{\onto}{\twoheadrightarrow}
\renewcommand{\:}{\colon}
\newcommand{\<}{\langle}
\renewcommand{\>}{\rangle}

\newcommand{\ldY}[1]{\overline{#1}}
\newcommand{\ldYdZ}[1]{\overline{\overline{#1}}}

\begin{document}

\begin{abstract}
We show that if we factor the long exact sequence in cohomology of
a cofiber sequence of spectra into short exact sequences, then the
$d_2$-differential in the Adams spectral sequence of any one term is
related in a precise way to Yoneda composition with the 2-extension
given by the complementary terms in the long exact sequence.  We use
this to give a complete analysis of the Adams spectral sequence for
the connective image-of-$J$ spectrum, finishing a calculation that
was begun by D.~Davis in~1975.
\end{abstract}

\maketitle

\section{Introduction}

Let
\begin{equation} \label{eq:XYZ}
X \overset{f}\longto Y \overset{g}\longto Z \overset{h}\longto \Sigma X
\end{equation}
be a homotopy cofiber sequence of spectra.  Let~$p$ be a prime, let $H
= H\bF_p$ be the mod~$p$ Eilenberg--MacLane ring spectrum, let $H^*$
denote mod~$p$ cohomology, and consider the induced long exact sequence
of $A$-modules
\begin{equation} \label{eq:les}
\cdots \from H^*(X)
       \overset{f^*}\longfrom H^*(Y)
       \overset{g^*}\longfrom H^*(Z)
       \overset{h^*}\longfrom \Sigma H^*(X)
	\from \cdots \,,
\end{equation}
where $A = H^*(H)$ denotes the mod~$p$ Steenrod algebra.  Letting
$$
K = \ker(g^*) \qquad
I = \im(g^*) \qquad
C = \cok(g^*)
$$
we obtain the following short exact sequences of $A$-modules:
\begin{align*}
0 \from \Sigma^{-1} K &\overset{q}\longfrom H^*(X)
	\overset{i}\longfrom C \from 0
	\tag{$E_X$} \\
0 \from C &\overset{q_Y}\longfrom H^*(Y) 
        \overset{i_Y}\longfrom I \from 0
	\tag{$E_Y$} \\
0 \from I &\overset{q_Z}\longfrom H^*(Z) 
        \overset{i_Z}\longfrom K \from 0
	\tag{$E_Z$} \,.
\end{align*}

\begin{figure}
$$
\xymatrix{
E_2^{s,t}(\Sigma^{-1} Z) \ar[d] \ar[r]^-{d_2} \ar@/_5pc/[dd]_-{h} 
	& E_2^{s+2,t+1}(\Sigma^{-1} Z) \ar[d] \ar@/^5pc/[dd]^-{h} \\
\Ext_A^{s,t}(\Sigma^{-1} K, \bF_p) \ar[d]^-{q^*}
	& \Ext_A^{s+2,t+1}(\Sigma^{-1} K, \bF_p) \ar[d]^-{q^*} \\
E_2^{s,t}(X) \ar[d]^-{i^*} \ar[r]^-{d_2} \ar@/_5pc/[dd]_-{f}
	& E_2^{s+2,t+1}(X) \ar[d]^-{i^*} \ar@/^5pc/[dd]^-{f} \\
\Ext_A^{s,t}(C, \bF_p) \ar[d]
	& \Ext_A^{s+2,t+1}(C, \bF_p) \ar[d] \\
E_2^{s,t}(Y) \ar[r]^-{d_2}
	& E_2^{s+2,t+1}(Y)
}
$$
\caption{$\Ext$-groups and $d_2$-differentials\label{E2d2ZXY}}
\end{figure}

We can compare the Adams $d_2$-differentials for $X$, $Y$ and~$Z$,
as in Figure~\ref{E2d2ZXY}.  If $g^*$ is injective (so that $K =
0$) or surjective (so that $C = 0$), the composite $i^* d_2 q^*$ is
necessarily zero.  If $g^*$ is trivial (so that $I = 0$), then $i = f^*$
and $q = h^*$, and $i^* d_2 q^*$ factors through $f h$, hence is zero.
In these cases the Adams $E_2$-terms of $X$, $Y$ and $Z$ form a long exact
sequence, and the geometric boundary theorem~\cite{Bru78} gives useful
information about the relationship between their $d_2$-differentials.

Our first main result address what happens when $g^*$ is neither injective,
surjective or zero, in which case $i^* d_2 q^*$ can indeed be nontrivial.

\begin{theorem} \label{thm:id2q}
Consider the homotopy cofiber sequence~\eqref{eq:XYZ}.  In the mod~$p$
Adams spectral sequence for~$X$ the composite homomorphism
$$
i^* d_2 q^* \: \Ext_A^{s,t}(\Sigma^{-1} K, \bF_p)
	\overset{q^*}\longto E_2^{s,t}(X)
	\overset{d_2}\longto E_2^{s+2,t+1}(X)
	\overset{i^*}\longto \Ext_A^{s+2,t+1}(C, \bF_p)
$$
agrees, up to a sign, with the composition
$$
\delta_Y \delta_Z \: \Ext_A^{s,t}(\Sigma^{-1} K, \bF_p)
	\overset{\delta_Z}\longto \Ext_A^{s+1,t+1}(I, \bF_p)
	\overset{\delta_Y}\longto \Ext_A^{s+2,t+1}(C, \bF_p)
$$
of the connecting homomorphisms associated to the extensions $(E_Z)$
and~$(E_Y)$.
\end{theorem}

Furthermore, $\delta_Y \delta_Z$ agrees, up to a sign, with the Yoneda
product with $e_Z e_Y$, where $e_Z \in \Ext_A^{1,0}(I, K)$, $e_Y \in
\Ext_A^{1,0}(C, I)$ and $e_Z e_Y \in \Ext_A^{2,0}(C, K)$ are the classes
associated to the extensions $(E_Z)$, $(E_Y)$ and
$$
0 \from C \longfrom H^*(Y) \overset{g^*}\longfrom H^*(Z)
	\longfrom K \from 0 \,,
$$
respectively.  See \cite{Mac63}*{Thm.~III.9.1}.
Note that, while the Adams $d_2$-differential in $E_2(X)$ generally
depends on the topological structure of the spectrum~$X$, our theorem
shows that the composite $i^* d_2 q^*$ only depends on the algebraic
structure given by the long exact sequence~\eqref{eq:les} of $A$-modules.

We thank the referee for suggesting that our arguments for
Theorem~\ref{thm:id2q} carry through as written for the $E$-based
Adams--Novikov spectral sequence, for a homotopy commutative ring
spectrum~$E$ with $E_* E$ flat over~$E_*$, replacing the spectrum $H$
by~$E$.   See Remark~\ref{rem:AdamsNovikov}.

We apply our first theorem to the special case of the (implicitly
$2$-complete) homotopy (co-)fiber sequence
\begin{equation} \label{eq:j}
j \longto ko \overset{\psi}\longto \Sigma^4 ksp \longto \Sigma j
\end{equation}
defining the connective image-of-$J$-spectrum~$j$.  Here $ko$ and $ksp$
denote connective real and quaternionic $K$-theory, respectively, and
$\psi$ denotes a lift of $\psi^3-1 \: ko \to ko$ over the $3$-connected
cover $bspin \simeq \Sigma^4 ksp$.  The homotopy groups of $ko$
and $ksp$, and the homomorphism induced by $\psi$, are well known,
and the resulting homotopy groups of~$j$ are easily calculated.

The full structure of the Adams spectral sequence converging to $\pi_*(j)$
is not as easy to determine.  In 1975, Davis~\cite{Dav75}
obtained the extension
\begin{equation*}
0 \from \Sigma^{-1} K \overset{q}\longfrom H^*(j)
	\overset{i}\longfrom C \from 0
	\tag{$E_j$}
\end{equation*}
with
\begin{align*}
C &\cong A/A(Sq^1, Sq^2, Sq^4) = A/\!/A(2) \\
K &\cong \Sigma^8 A/A(Sq^1, Sq^7, Sq^4 Sq^6 + Sq^6 Sq^4) \,,
\end{align*}
and the long exact sequence
$$
\cdots \to \Ext_A^{s-1,t}(C, \bF_2)
	\overset{\delta_X}\longto \Ext_A^{s,t}(\Sigma^{-1} K, \bF_2)
	\overset{q^*}\longto E_2^{s,t}(j)
	\overset{i^*}\longto \Ext_A^{s,t}(C, \bF_2)
\to \cdots \,,
$$
leading to the rather complex Adams $E_2$-term shown in
Figure~\ref{fig:E2d2j}.  Davis did not, however, determine
the differential pattern leading from $E_2(j)$ to the much
simpler abutment $\pi_*(j)$.  We can now resolve this problem,
by using Theorem~\ref{thm:id2q} to get a lower bound on the rank
of $d_2 \: E_2^{s,t}(j) \to E_2^{s+2,t+1}(j)$, and thereby deduce
Theorem~\ref{thm:E3j}.  The $(E_2,d_2)$-, $E_3$- and $E_\infty$-terms
for~$j$ are illustrated in Figures~\ref{fig:E2d2j}, \ref{upper}
and~\ref{infty}, respectively.  Our theorem shows that almost all of the
drastic amount of cancellation that must occur in this spectral sequence
occurs from~$E_2(j)$ to~$E_3(j)$, with the $d_2$-differential reducing
the Krull dimension from~$4$ to~$2$.

\begin{theorem} \label{thm:E3j}
Let $j$ be the connective image-of-$J$ spectrum at $p=2$.
There is an isomorphism
\begin{align*}
E_3(j) &\cong
	\bF_2[w_1] \{h_1, h_1^2, h_2, h_0 h_2, h_0^2 h_2, c_0, h_1 c_0\} \\
&\quad\oplus
	\bF_2[w_1^2] \{h_0^i h_3 \mid 0 \le i \le 3\} \\
&\quad\oplus
	\left(\bF_2[h_0, w_1^4] \{1, h_3 w_1^3\}
	\oplus \bF_2[w_1^4] \{h_0^i h_3 w_1 \mid 0 \le i \le 4\}\right) \,,
\end{align*}
with generators in $(s,t)$-bidegrees $|h_0| = (1,1)$, $|h_1| = (1,2)$,
$|h_2| = (1,4)$, $|h_3| = (1,8)$, $|c_0| = (3,11)$ and $|w_1| = (4,12)$.
The remaining nonzero differentials are
$$
d_r(h_0^i w_1^k) = h_0^{i+r+3} h_3 w_1^{k-1}
$$
for $r\ge3$, $i\ge0$ and~$\ord_2(k) = r-1$.
Hence
\begin{align*}
E_\infty(j) &\cong \bF_2[h_0] \\
&\quad\oplus
        \bF_2[w_1] \{h_1, h_1^2, h_2, h_0 h_2, h_0^2 h_2, c_0, h_1 c_0\} \\
&\quad\oplus
	\bigoplus_{r\ge1}
        \bF_2[w_1^{2^r}] \{h_0^i h_3 w_1^{2^{r-1}-1} \mid 0 \le i \le r+2\} \,.
\end{align*}
\end{theorem}

Here and below we use $\ord_p(k)$ to denote the $p$-adic valuation of~$k$.

\begin{remark}
We write $h_3 \in E_2^{1,8}(j)$ for the image of the usual class $h_3
\in E_2^{1,8}(S)$ under the unit map $e \: S \to j$.  This is also the
image $q^*(\Sigma^{-1} h_0\ldYdZ{h_2^2}) = q^*(i_Z^*(\Sigma^3 v'))$ for classes
$\Sigma^{-1} h_0\ldYdZ{h_2^2} \in \Ext_A^{1,8}(\Sigma^{-1} K, \bF_2)$ and
$\Sigma^3 v' \in E_2^{1,8}(\Sigma^3 ksp)$.  There is a well-defined 
action by powers of~$w_1$ on the latter two $\Ext$-groups.  The classes 
$h_3 w_1^{k-1}$ in $E_2(j)$ are therefore defined to be the images 
under~$q^*i_Z^*$ of the products $\Sigma^3 v' \cdot w_1^{k-1}$. 
We explain the notation~$\ldYdZ{x}$ in Lemma~\ref{lem:I-Hksp-K}, which
builds upon a notation~$\ldY{x}$ introduced in Lemma~\ref{lem:C-Hko-I}.

The unit map sends $\bF_2[h_0] \subset E_\infty(S)$ isomorphically
to $\bF_2[h_0] \subset E_\infty(j)$.  Direct calculation with
$\Ext_A$ shows that each class
$x \in \{h_1, h_1^2, h_2, h_0 h_2, h_0^2 h_2, c_0, h_1 c_0\}$
in $E_2(S)$ is mapped by~$e$ to the class with the same name in~$E_2(j)$.
Naturality with respect to the Adams periodicity operator~$Px = \<h_3,
h_0^4, x\>$ then shows that $P^{k-1} x$ maps to $x w_1^{k-1}$, for each
such~$x$ and $k\ge1$.  The classes $P^{k-1} x$ are all infinite cycles
by Adams vanishing~\cite{Ada66}*{Thm.~1.1}, and Theorem~\ref{thm:E3j}
shows that the classes $x w_1^{k-1}$ remain nonzero in~$E_\infty(j)$.
Hence $P^{k-1} x$ and $x w_1^{k-1}$ all survive to nonzero classes in
the $E_\infty$-terms, with $e$ mapping the former to the latter.
\end{remark}

\begin{remark}
The case of topological degree~$8k-1$ is well known to be significantly
more difficult.
Davis and Mahowald~\cite{DM89}*{Thm.~1.1} proved that any generator
$\rho_{8k-1}$ of the image of the $J$-homomorphism in~$\pi_{8k-1}(S)$ is
detected in $E_\infty(S)$ by a class of Adams filtration $4k-3-\ord_2(k)$,
and that this class supports an $h_0$-tower that ends in Adams
filtration~$4k$.  On the other hand, Theorem~\ref{thm:E3j} shows that
any generator $j_{8k-1}$ of $\pi_{8k-1}(j)$ is detected in $E_\infty(j)$
by $h_3 w_1^{k-1}$ in Adams filtration~$4k-3$, and that this supports
an~$h_0$-tower of the same height, ending in Adams filtration~$4k +
\ord_2(k)$.  It follows that $e$ maps the image-of-$J$ subgroup in
$\pi_{8k-1}(S)$ isomorphically to $\pi_{8k-1}(j)$, but increases the
Adams filtration of each class by exactly~$\ord_2(k)$.

The hidden $\eta$-extension in $E_\infty(j)$ from $h_3 w_1^{k-1}$
detecting~$j_{8k-1}$ to $c_0 w_1^{k-1}$ detecting~$\eta j_{8k-1}$
increases Adams filtration by~$2$ in each case, see Figure~\ref{infty}.
This is in contrast to the hidden $\eta$-extension in $E_\infty(S)$
from the class detecting~$\rho_{8k-1}$ to $P^{k-1} c_0$ detecting~$\eta
\rho_{8k-1}$ (except for $k=1$), which shifts Adams filtration by the
variable amount $2 + \ord_2(k)$.
\end{remark}

In Section~\ref{sec:jmod2} we make a similar analysis of the
Adams spectral sequence for the mod~$2$ reduction~$j/2$ of the
image-of-$J$ spectrum, which turns out to collapse at the $E_3$-term.
In Section~\ref{sec:jpodd} we summarize the corresponding calculations for
the $p$-primary image-of-$J$ spectrum, where $p$ is any odd prime.  See
Theorems~\ref{thm:E3jmod2} and~\ref{thm:E3jpodd} for precise statements,
and Figures~\ref{fig:jmod2} and~\ref{fig:E2d2jp3} for illustrations.

\section{Adams differentials in a homotopy cofiber sequence}

Our study depends on a standard functorial construction of the Adams
spectral sequence, which we now review.  Let
$$
\dots \to S_{s+1}
	\overset{\alpha}\longto S_s
	\to \dots
	\to S_1
	\overset{\alpha}\longto S_0 = S
$$
be the canonical mod~$p$ Adams tower for the sphere spectrum.
For each $s\ge0$ we have a homotopy cofiber sequence
$$
S_{s+1} \overset{\alpha}\longto S_s \overset{\beta}\longto
	H \wedge S_s \overset{\gamma}\longto \Sigma S_{s+1} \,,
$$
where $\beta$ is induced by the ring spectrum unit map $S \to H$.
Smashing these diagrams with~$X$ we obtain the canonical Adams tower
for~$X$, and we set $X_s = S_s \wedge X$.  The mod~$p$ Adams spectral
sequence for~$X$ is the homotopy spectral sequence associated to the
following diagram, where $\gamma$ has degree~$-1$.
$$
\xymatrix{
\cdots \ar[r]
	& X_{s+1} \ar[r]^-{\alpha}
	& X_s \ar[r] \ar[d]^-{\beta}
	& \cdots \ar[r]
	& X_1 \ar[r]^-{\alpha}
	& X \ar[d]^-{\beta} \\
& & H \wedge X_s \ar@{-->}[ul]^-{\gamma}
& & & H \wedge X \ar@{-->}[ul]^-{\gamma}
}
$$
When $H^*(X)$ is bounded below and of finite type, the K{\"u}nneth
theorem shows that the cohomology of the lower part is isomorphic to
a free resolution
$$
\dots \to C_s(A, H^*(X))
	\overset{\partial_s}\longto C_{s-1}(A, H^*(X))
	\to \dots \to C_0(A, H^*(X))
	\overset{\epsilon}\longto H^*(X) \,,
$$
where we set
$$
C_s(A, M) = A \otimes I(A)^{\otimes s} \otimes M
$$
for any $A$-module~$M$.  As usual, $I(A) = \ker(\epsilon \: A \to \bF_p)$
denotes the augmentation ideal.  The Hopf algebra $A$ acts diagonally
on this $(s+2)$-fold tensor product, which is nonetheless free because
of the untwisting isomorphism.  Furthermore, the Hurewicz homomorphism
induces isomorphisms
$$
\pi_{t-s}(H \wedge X_s)
	\overset{\cong}\longto
	\Hom_A(H^*(H \wedge X_s), \Sigma^{t-s} \bF_p)
	\cong C_A^{s,t}(H^*(X), \bF_p) \,,
$$
where
$$
C_A^{s,t}(M, \bF_p) = \Hom_A(C_s(A, M), \Sigma^t \bF_p) \,.
$$
By naturality, it follows that
$$
(E_1^{s,t}(X), d_1) \cong (C_A^{s,t}(H^*(X), \bF_p), \delta)
$$
with $\delta = \Hom(\partial, 1)$, which leads to the familiar formula
$$
E_2^{s,t}(X) \cong \Ext_A^{s,t}(H^*(X), \bF_p) \,.
$$

\begin{remark}
The canonical resolution $C_*(A, M) \to M$ is isomorphic to the normalized
bar resolution, and $C_A^*(M, \bF_p)$ is isomorphic to the normalized
cobar complex, but we do not need to make these isomorphisms explicit.
\end{remark}

\begin{remark}
\label{rem:AdamsNovikov}
If $H^*(X)$ is not bounded below and of finite type, we can instead work
with $H_*(X)$ as an $A_*$-comodule, where $A_* = H_*(H)$ is the dual
mod~$p$ Steenrod (Hopf) algebra.  We set
$$
C_s(A_*, M_*) = A_* \otimes J(A_*)^{\otimes s} \otimes M_*
$$
and
$$
C_{A_*}^{s,t}(\bF_p, M_*) = \Hom_{A_*}(\Sigma^t \bF_p, C_s(A_*, M_*)) \,,
$$
where $M_*$ is any $A_*$-comodule and $J(A_*) = \cok(\eta \: \bF_p \to
A_*)$.  Then
$$
E_1^{s,t}(X) = \pi_{t-s}(H \wedge X_s)
	\cong C_{A_*}^{s,t}(\bF_p, H_*(X))
$$
and $E_2^{s,t}(X) \cong \Ext_{A_*}^{s,t}(\bF_p, H_*(X))$, with $\Ext$
formed in $A_*$-comodules.  Our results all apply in this generality,
but for ease of comparison with the calculations in the next section we
prefer to write in terms of cohomology and $A$-modules.

Our arguments then carry through as written for the $E$-based
Adams--Novikov spectral sequence, for a homotopy commutative ring
spectrum~$E$ with $E_* E$ flat over~$E_*$, replacing the spectrum $H$
by~$E$, the Hopf algebra $A_*$ by the Hopf algebroid $(E_* E, E_*)$,
and the $A_*$-comodule $H_*(X)$ by the $E_* E$-comodule $E_*(X)$.
\end{remark}

The Adams differential $d_2 \: E_2^{s,t}(X) \to E_2^{s+2,t+1}(X)$ is given
by $d_2([x]) = [\beta(u)]$, where $x \in E_1^{s,t}(X)$ is a $d_1$-cycle,
$u \in \pi_{t-s-1}(X_{s+2})$, and $\gamma(x) = \alpha(u)$.
$$
\xymatrix{
\pi_*(X_{s+2}) \ar[r]^-{\alpha} \ar[d]^-{\beta}
	& \pi_*(X_{s+1}) \ar[r]^-{\alpha} \ar[d]^-{\beta}
	& \pi_*(X_s) \ar[d]^-{\beta} \\
E_1^{s+2,*}(X) & E_1^{s+1,*}(X) & E_1^{s,*}(X) \ar@{-->}[ul]^-{\gamma}
}
$$
Let $S_{s,r} = \cof(\alpha^r \: S_{s+r} \to S_s)$ be the mapping cone, so
that $S_{s,1} \simeq H \wedge S_s$ and we have homotopy cofiber sequences
\begin{align}
S_{s+2,1} &\overset{\alpha'}\longto S_{s+1,2}
	\overset{\beta'}\longto S_{s+1,1}
	\overset{\gamma'}\longto \Sigma S_{s+2,1} \label{eq:prime} \\
S_{s+1,2} &\overset{\alpha''}\longto S_{s,3}
	\overset{\beta''}\longto S_{s,1}
	\overset{\gamma''}\longto \Sigma S_{s+1,2} \,.
\end{align}
Letting $X_{s,r} = S_{s,r} \wedge X$ we get similar sequences by
smashing these with~$X$.
Then $d_2([x]) = [v]$ for any $v \in E_1^{s+2,t+1}(X)$ with $\gamma''(x)
= \alpha'(v)$ in the following diagram.
$$
\xymatrix{
E_1^{s+2,*}(X) \ar[r]^-{\alpha'}
	& \pi_*(X_{s+1,2})
	& E_1^{s,*}(X) \ar@{-->}[l]_-{\gamma''}
}
$$
In each case, a choice of $u$ or $v$ exists because $\beta \gamma(x)
= d_1(x) = 0$, and the difference between any two choices maps to zero
in $E_2^{s+2,t+1}(X)$.

\begin{figure}
$$
\xymatrix{
Y_{s+2,1} \ar[rr]^-{\alpha'} \ar[dd]^-{g}
	&& Y_{s+1,2} \ar[r]^-{\beta'} \ar[dd]^-{g}
	& Y_{s+1,1} \ar[r]^-{\gamma'} \ar[dd]^-{g}
	& \Sigma Y_{s+2,1} \ar[dd]^-{g} \\
& \Sigma^{-1} Z_{s,1} \ar[dr]^-{\gamma''} \ar[dd]|{\hole}^(0.3){h} \\
Z_{s+2,1} \ar[rr]^(0.3){\alpha'} \ar[dd]^-{h}
	&& Z_{s+1,2} \ar[r]^-{\beta'} \ar[dd]^-{h}
	& Z_{s+1,1} \ar[r]^-{\gamma'} \ar[dd]^-{h}
	& \Sigma Z_{s+2,1} \ar[dd]^-{h} \\
& X_{s,1} \ar[dr]^-{\gamma''} \\
\Sigma X_{s+2,1} \ar[rr]^-{\alpha'} \ar[d]^-{f}
	&& \Sigma X_{s+1,2} \ar[r]^-{\beta'} \ar[d]^-{f}
	& \Sigma X_{s+1,1} \ar[r]^-{\gamma'} \ar[d]^-{f}
	& \Sigma^2 X_{s+2,1} \ar[d]^-{f} \\
\Sigma Y_{s+2,1} \ar[rr]^-{\alpha'}
	&& \Sigma Y_{s+1,2} \ar[r]^-{\beta'}
	& \Sigma Y_{s+1,1} \ar[r]^-{\gamma'}
	& \Sigma^2 Y_{s+2,1}
}
$$
\caption{Horizontal and vertical homotopy cofiber sequences \label{horvert}}
\end{figure}

Smashing the continuation of~\eqref{eq:XYZ} and~\eqref{eq:prime}
together, we obtain the commutative diagram in Figure~\ref{horvert}
with horizontal and vertical homotopy cofiber sequences, up to some
signs which we suppress.
Applying $\pi_*$, we obtain a similar diagram with long exact rows
and columns.

\begin{proposition} \label{prop:pull-push}
Let $w \in \pi_n(\Sigma X_{s+1,2})$ satisfy $f(w) = 0$ and
$\beta'(w) = 0$.  Then
$$
\pm f (\alpha')^{-1} (w) = \gamma' g^{-1} \beta' h^{-1}(w)
$$
as subsets of $\pi_n(\Sigma Y_{s+2,1})$.
The indeterminacy of either expression is the image of $\pm f \gamma' =
\gamma' f$.
\end{proposition}

\begin{proof}
If $X$, $Y$ and (hence) $Z$ are dualizable, then this follows from
the fill-in axiom for the triangulated structure on the stable
homotopy category, applied to the diagram below.
$$
\xymatrix{
\Sigma^{n-1} DX \ar[r]^-{\pm Dh} \ar[d] \ar[dr]^-{w}
	& \Sigma^n DZ \ar[r]^-{Dg} \ar[d]
	& \Sigma^n DY \ar[r]^-{Df} \ar@{..>}[d]
	& \Sigma^n DX \ar[d] \\
S_{s+2,1} \ar[r]^-{\alpha'}
	& S_{s+1,2} \ar[r]^-{\beta'}
	& S_{s+1,1} \ar[r]^-{\gamma'}
	& \Sigma S_{s+2,1}
}
$$
The general case follows by a passage to colimits.  See
also \cite{BG95}*{Lem.~2.2}, and~\cite{AM17}*{Lem.~9.3.2}
correcting~\cite{Mil81}*{Lem.~6.7}.
\end{proof}

\begin{proof}[Proof of Theorem~\ref{thm:id2q}]
The extensions~$(E_X)$, $(E_Y)$ and~$(E_Z)$ induce short exact sequences of
cochain complexes, giving factorizations of the maps of Adams $E_1$-terms
induced by $f$, $g$ and~$h$, as shown in Figure~\ref{chase}.

\begin{figure}
$$
\xymatrix@C-1.8pc{
& & & E_1^{s+1,t+1}(Y) \ar[r]^-{\gamma'} \ar@{->>}[d]_-{i_Y^*}
	\ar@/^4pc/[dd]^-{g}
	& E_1^{s+2,t+1}(Y) \\
& E_1^{s,t+1}(Z) \ar[dr]^-{\gamma''} \ar@{->>}[d]_-{i_Z^*}
	& & C_A^{s+1,t+1}(I, \bF_p) \ar@{ >->}[d]_-{q_Z^*} \\
& C_A^{s,t}(\Sigma^{-1} K, \bF_p) \ar@{ >->}[d]_-{q^*}
	& \pi_{t-s}(Z_{s+1,2}) \ar[dd]^-{h} \ar[r]^-{\beta'}
	& E_1^{s+1,t+1}(Z) \\
& E_1^{s,t}(X) \ar[dr]^-{\gamma''} \\
E_1^{s+2,t+1}(X) \ar[rr]^-{\alpha'} \ar@{->>}[d]_-{i^*}
	\ar@/^4pc/[dd]^-{f}
	& & \pi_{t-s}(\Sigma X_{s+1,2}) \\
C_A^{s+2,t+1}(C, \bF_p) \ar@{ >->}[d]_-{q_Y^*} \\
E_1^{s+2,t+1}(Y)
}
$$
\caption{Diagram for the proof of Theorem~\ref{thm:id2q} \label{chase}}
\end{figure}

Let $x \in C_A^{s,t}(\Sigma^{-1} K, \bF_p)$ be a cocycle, representing
an arbitrary cohomology class $[x] \in \Ext_A^{s,t}(\Sigma^{-1} K,
\bF_p)$.  Then $q^* x \in E_1^{s,t}(X)$ is a $d_1$-cycle, with class
$[q^* x] = q^* [x] \in E_2^{s,t}(X)$.  The differential $d_2 q^* [x]
\in E_2^{s+2,t+1}(X)$ is the class $[v]$ of any $v \in E_1^{s+2,t+1}(X)$
with $\alpha' v = w = \gamma'' q^* x$, where $w \in \pi_{t-s}(\Sigma
X_{s+1,2})$.  Hence $i^* d_2 q^* [x] \in \Ext_A^{s+2,t+1}(C, \bF_p)$
is the class $[i^* v]$ of $i^* v \in C_A^{s+2,t+1}(C, \bF_p)$.

The image $q_Y^* i^* v = f v \in E_1^{s+2,t+1}(Y)$ of $i^* v$ is thus an
element of $f (\alpha')^{-1} w$.  We find another element in the same
coset by applying Proposition~\ref{prop:pull-push} to~$w$.  First choose
an $\tilde x \in E_1^{s,t+1}(Z)$ that maps by~$i_Z^*$ to~$x$.  Then $h \gamma''
\tilde x = w = \alpha' v$, so $w \in \im(h) = \ker(f)$ and $w \in
\im(\alpha') = \ker(\beta')$.  Hence the hypothesis of the proposition
is satisfied, and we can calculate $f (\alpha')^{-1}(w)$ as $\gamma'
g^{-1} \beta' h^{-1}(w)$, up to a possible sign.

We have $\gamma'' \tilde x \in h^{-1}(w)$, so $d_1(\tilde x)
= \beta' \gamma'' \tilde x \in \beta' h^{-1}(w)$.  Since $x$
is a cocycle, $d_1(\tilde x)$ lifts uniquely over~$q_Z^*$ to a cocycle
$z \in C_A^{s+1,t+1}(I, \bF_p)$, as in the definition of
the connecting homomorphism
$$
\delta_Z \: \Ext_A^{s,t}(\Sigma^{-1} K, \bF_p)
	\longto \Ext_A^{s+1,t+1}(I, \bF_p)
$$
associated to the short exact sequence of cochain complexes induced
by~$(E_Z)$.  Hence $[z] = \delta_Z [x]$.

Next choose a $\tilde z \in E_1^{s+1,t+1}(Y)$ that maps by~$i_Y^*$ to~$z$.
Then $g \tilde z = d_1(\tilde x)$, so $\tilde z \in g^{-1} \beta'
h^{-1}(w)$.  The image $\gamma' \tilde z = d_1(\tilde z) = \bar y$
in $E_1^{s+2,t+1}(Y)$ then lies in the coset $\gamma' g^{-1} \beta' h^{-1}(w)$.
Moreover, since $z$ is a cocycle, $\bar y$ lifts uniquely over~$q_Y^*$ to
a cocycle $y \in C_A^{s+2,t+1}(C, \bF_p)$, as in the definition of the
connecting homomorphism
$$
\delta_Y \: \Ext_A^{s+1,t+1}(I, \bF_p)
	\longto \Ext_A^{s+2,t+1}(C, \bF_p)
$$
associated to~$(E_Y)$.  Hence $[y] = \delta_Y [z]$.

We have now shown that $i^* d_2 q^* [x]$ and $\delta_Y \delta_Z [x]$ in
$\Ext_A^{s+2,t+1}(C, \bF_p)$ are represented by cocycles $i^* v$ and $y$
in $C_A^{s+2,t+1}(C, \bF_p)$, respectively, which have the same image in
$E_1^{s+2,t+1}(Y)$, up to a sign and modulo the image of $f \gamma'$.
In view of the commutative diagram
$$
\xymatrix{
E_1^{s+1,t+1}(X) \ar[r]^-{\gamma'} \ar@{->>}[d]_-{i^*}
	& E_1^{s+2,t+1}(X) \ar@{->>}[d]_-{i^*} 
	\ar@/^4pc/[dd]^-{f} \\
C_A^{s+1,t+1}(C, \bF_p) \ar[r]^-{\delta} \ar@{ >->}[d]_-{q_Y^*}
	& C_A^{s+2,t+1}(C, \bF_p) \ar@{ >->}[d]_-{q_Y^*} \\
E_1^{s+1,t+1}(Y) \ar[r]^-{\gamma'}
	& E_1^{s+2,t+1}(Y)
}
$$
it follows that the cocycles $i^* v$ and $y$ agree in
$C_A^{s+2,t+1}(C, \bF_p)$, up to a sign and modulo
the image of $\delta$, i.e., up to a coboundary.  In other
words, $[i^*v] = \pm [y]$ in $\Ext_A^{s+2,t+1}(C, \bF_p)$.
\end{proof}

\section{The image-of-$J$ spectrum}

We now specialize to the case when the homotopy cofiber
sequence~\eqref{eq:XYZ} is the sequence~\eqref{eq:j} defining the
connective image-of-$J$ spectrum, implicitly completed at the prime
$p=2$.  The mod~$2$ cohomology and Adams spectral sequence for the ring
spectrum~$ko$ and its module spectrum~$ksp$ are well-known.
Let $A(n) \subset A$ be the subalgebra generated by $Sq^1, \dots,
Sq^{2^n}$.

\begin{proposition}[\cite{Sto63}, \cite{BR21}*{\S2.6}]\label{prop:cohko}
\leavevmode
\begin{enumerate}
\item
$H^*(ko) = A/A(Sq^1, Sq^2) = A/\!/A(1)$ and
$$
E_2(ko) = \Ext_{A(1)}(\bF_2, \bF_2)
	= \bF_2[h_0, h_1, v, w_1]/(h_0 h_1, h_1^3, h_1 v, v^2 + h_0^2 w_1)
$$
with algebra generators in $(s,t)$-bidegrees $|h_0| = (1,1)$, $|h_1| =
(1,2)$, $|v| = (3,7)$ and $|w_1| = (4,12)$.

\item
$\pi_*(ko) = \bZ[\eta, A, B]/(2 \eta, \eta^3, \eta A, A^2 - 4 B)$
with $|\eta| = 1$, $|A| = 4$ and $|B| = 8$.

\item
$H^*(ksp) = A/A(Sq^1, Sq^2 Sq^3)$ and
$$
E_2(ksp) = E_2(ko)\{1, v'\} / (h_1 \cdot 1, v \cdot 1 + h_0^2 \cdot v',
	v \cdot v' + w_1 \cdot 1)
$$
with module generators in $(s,t)$-bidegrees $|1| = (0,0)$ and $|v'| =
(1,5)$.
\end{enumerate}
\end{proposition}

The action of the Adams operation $\psi^3$ on the homotopy of $ko$,
and the resulting homotopy groups of the image-of-$J$ spectrum, are
also well-known.

\begin{proposition}[\cite{Ada62}, \cite{Dav75}*{Prop.~2},
	\cite{BR21}*{\S11.3}] \label{prop:pij}
\leavevmode
\begin{enumerate}
\item
$\psi^3(\eta) = \eta$, $\psi^3(A) = 3^2 A$ and $\psi^3(B) = 3^4 B$.
\item
$\pi_0(j) = \bZ$, $\pi_1(j) = \bZ/2\{\eta\}$ and, when $n\ge2$,
$$
\pi_n(j) = \begin{cases}
\bZ/16 k & \text{for $n = 8k-1$,} \\
\bZ/2 & \text{for $n \equiv 0 \mod 8$,} \\
(\bZ/2)^2 & \text{for $n \equiv 1 \mod 8$,} \\
\bZ/2 & \text{for $n \equiv 2 \mod 8$,} \\
\bZ/8 & \text{for $n \equiv 3 \mod 8$,} \\
0 & \text{otherwise.}
\end{cases}
$$
\end{enumerate}
\end{proposition}

The $A$-module homomorphism $\psi^*$ is induced up from
the subalgebra~$A(2)$, so the calculation of its kernel, image
and cokernel is a finite algebraic problem.

\begin{proposition}[\cite{Dav75}*{Lem.~3}, \cite{AR05}*{Lem.~7.6(c)}]
\label{prop:KICStructure}
\leavevmode
\begin{enumerate}
\item
$\psi^* \: H^*(\Sigma^4 ksp) \to H^*(ko)$ is determined by
$\psi^*(\Sigma^4 1) = Sq^4$.
\item
$K = \ker(\psi^*) = \Sigma^8 A/A(Sq^1, Sq^7, Sq^4 Sq^6 + Sq^6 Sq^4)$.
\item
$I = \im(\psi^*) = \Sigma^4 A/A(Sq^1, Sq^4)$.
\item
$C = \cok(\psi^*) = A/A(Sq^1, Sq^2, Sq^4) = A /\!/ A(2)$.
\end{enumerate}
\end{proposition}

\begin{proof}
Only case~(3) may be new.  It amounts to the claim that the
$A(2)$-submodule of $A(2)/\!/A(1)$ generated by $Sq^4$ has annihilator
ideal $A(2)(Sq^1, Sq^4)$, which can be checked by direct calculation.
This is also implicit in \cite{DM82}*{Thm.~5.9}.
\end{proof}

We will see that $\Ext_A(K, \bF_2)$, $\Ext_A(I, \bF_2)$ and $\Ext_A(C,
\bF_2)$ are closely related, where the latter is explicitly known.

\begin{proposition}[\cite{SI67}, \cite{BR21}*{\S3.5}] \label{prop:ExtA2}
\begin{align*}
\Ext_A(C, \bF_2) &= \Ext_{A(2)}(\bF_2, \bF_2) \\
&= \bF_2[h_0, h_1, h_2, c_0, w_1, \alpha, \beta, d_0, e_0, g, \gamma,
	\delta, w_2]/(\sim)
\end{align*}
is a free $\bF_2[w_1, w_2]$-module, where $(\sim) = (h_0 h_1, \dots,
\delta^2)$ denotes an ideal generated by~$54$ explicit relations.
The generators are graded as follows.
$$
\begin{tabular}{ >{$}c<{$}|>{$}c<{$} >{$}c<{$} >{$}c<{$} >{$}c<{$}
        >{$}c<{$} >{$}c<{$} >{$}c<{$} >{$}c<{$} >{$}c<{$} >{$}c<{$}
        >{$}c<{$} >{$}c<{$} >{$}c<{$} }
& h_0 & h_1 & h_2 & c_0 & w_1 & \alpha
        & \beta & d_0 & e_0 & g & \gamma & \delta & w_2 \\
\hline
s & 1 & 1 & 1 & 3 & 4 & 3
        & 3 & 4 & 4 & 4 & 5 & 7 & 8 \\
t & 1 & 2 & 4 & 11 & 12 & 15
        & 18 & 18 & 21 & 24 & 30 & 39 & 56 \\
t-s & 0 & 1 & 3 & 8 & 8 & 12
	& 15 & 14 & 17 & 20 & 25 & 32 & 48
\end{tabular}
$$
\end{proposition}

\begin{remark}
The $A$-modules $H^*(ko)$ and $H^*(ksp)$ are induced up from $A(1)$,
while the modules $K$, $I$ and~$C$, as well as the extensions~$(E_{ko})$
and~$(E_{\Sigma^4 ksp})$ below, are induced up from $A(2)$.  Hence the
long exact sequences in Lemmas~\ref{lem:C-Hko-I} and~\ref{lem:I-Hksp-K}
consist of $\Ext_{A(2)}(\bF_2, \bF_2)$-modules and -homomorphisms.
In particular, $w_1$ acts linearly on these sequences.  The module
$H^*(j)$ and the extension~$(E_j)$ are not induced up from~$A(2)$,
but from~$A(3)$, so the long exact sequence in the proof of
Theorem~\ref{thm:E3j} (below) is $\Ext_{A(3)}(\bF_2, \bF_2)$-linear.
It then follows from~\cite{Ada66}*{Lem.~4.4} that $w_1^2$ acts linearly
on that sequence.  Alternatively, we may use that $(E_j)$ is dual to
a square-zero extension of $A_*$-comodule algebras, which implies that
$\delta_X$ is a derivation.
\end{remark}

\begin{lemma} \label{lem:C-Hko-I}
The $A$-module extension
\begin{align*}
0 \from C \overset{q_Y}\longfrom H^*(ko)
	\overset{i_Y}\longfrom I \from 0
	\tag{$E_{ko}$}
\end{align*}
induces a long exact sequence
$$
\dots \to E_2^{s-1,t}(ko)
	\overset{i_Y^*}\longto \Ext^{s-1,t}_A(I, \bF_2)
	\overset{\delta_Y}\longto \Ext^{s,t}_A(C, \bF_2)
	\overset{q_Y^*}\longto E_2^{s,t}(ko)
\to \cdots
$$
with
$$
\bF_2[h_0, w_1] \{v\}
	\cong \ker(\delta_Y) \into \Ext_A^{*,*}(I, \bF_2)
$$
and
$$
\Ext_A^{*,*}(C, \bF_2) \onto \cok(\delta_Y)
	\cong \bF_2[h_0, w_1] \oplus \bF_2[w_1]\{h_1, h_1^2\} \,.
$$
For each nonzero $x \in \Ext_A^{s,t}(C, \bF_2)$ that maps trivially to
$\cok(\delta_Y)$ there is a unique lift $\ldY{x} \in \Ext_A^{s-1,t}(I, \bF_2)$
satisfying $\delta_Y(\ldY{x}) = x$, and $\Ext_A(I, \bF_2)$ consists
of these $\ldY{x}$, together with the free $\bF_2[h_0, w_1]$-module
on $i_Y^*(v) = h_0^3 \ldY{h_2}$.
\end{lemma}

\begin{proof}
The homomorphism $q_Y^* \: \Ext_A(C, \bF_2) \to E_2(ko)$ equals the
restriction homomorphism $\Ext_{A(2)}(\bF_2, \bF_2) \to \Ext_{A(1)}(\bF_2,
\bF_2)$, which sends $h_0$, $h_1$ and~$w_1$ to the elements with the
same names, and which sends the remaining algebra generators $h_2,
\dots, w_2$ to zero, because the corresponding target bidegrees
are trivial.  Hence this algebra homomorphism has image $\bF_2[h_0,
w_1] \oplus \bF_2[w_1]\{h_1, h_1^2\}$ and cokernel $\bF_2[h_0, w_1]
\{v\}$.  The uniqueness of the nonzero lifts $\ldY{x}$ follows from the fact that
$\Ext_A^{s,t}(C, \bF_2) = 0$ for those bidegrees in which $\ker(\delta_Y)$
is nonzero, that is, when $t-s = 8k+3$ and $s \geq 4k+4$.  This is an
immediate consequence of the presentation from \cite{BR21}*{Prop.~3.45}
of $\Ext_{A(2)}(\bF_2, \bF_2)$ as a direct sum of cyclic $R_0$-modules,
where $R_0 = \bF_2[g, w_1, w_2]$.  See also \cite{BR21}*{Fig.~3.12, 3.13}.
The relation $i_Y^*(v) = h_0 \cdot \ldY{h_0^2 h_2}$ holds, as can be
verified by direct calculation in $\Ext_A(I, \bF_2)$ or deduced as part
of the next proof.
\end{proof}

\begin{lemma} \label{lem:I-Hksp-K}
The $A$-module extension
\begin{align*}
0 \from I \overset{q_Z}\longfrom H^*(\Sigma^4 ksp)
	\overset{i_Z}\longfrom K \from 0
	\tag{$E_{\Sigma^4 ksp}$}
\end{align*}
induces a long exact sequence
$$
\dots \to E_2^{s-1,t}(\Sigma^4 ksp)
	\overset{i_Z^*}\longto \Ext^{s-1,t}_A(K, \bF_2)
	\overset{\delta_Z}\longto \Ext^{s,t}_A(I, \bF_2)
	\overset{q_Z^*}\longto E_2^{s,t}(\Sigma^4 ksp)
\to \cdots
$$
with
$$
\Sigma^4 \bF_2[h_0, w_1] \{v'\}
\cong \ker(\delta_Z) \into \Ext_A^{*,*}(K, \bF_2)
	$$
and
$$
\Ext_A^{*,*}(I, \bF_2) \onto \cok(\delta_Z) \cong
    \Sigma^4 \bigl( \bF_2[h_0, w_1] \oplus \bF_2[w_1]\{h_1 v', h_1^2 v'\} \bigr)\,.
$$
The submodule $\bF_2[h_0, w_1]\{\ldY{h_2}\} \oplus \bF_2[w_1]\{\ldY{c_0},
h_1 \ldY{c_0}\}$ of $\Ext_A(I, \bF_2)$ maps isomorphically to $\cok(\delta_Z)$.
For each nonzero $\ldY{x} \in \Ext_A^{s,t}(I, \bF_2)$ that maps trivially to
$\cok(\delta_Z)$ there is a unique lift $\ldYdZ{x} \in \Ext_A^{s-1,t}(K, \bF_2)$
satisfying $\delta_Z(\ldYdZ{x}) = \ldY{x}$, and $\Ext_A(K, \bF_2)$ consists
of these $\ldYdZ{x}$, together with the free $\bF_2[h_0, w_1]$-module
on $i_Z^*(\Sigma^4 v') = h_0 \ldYdZ{h_2^2}$.
\end{lemma}

\begin{proof}
We must prove that the image of
$$
q_Z^* \: \Ext_A(I, \bF_2) \longto E_2(\Sigma^4 ksp)
$$
is exactly $\Sigma^4 \bigl( \bF_2[h_0, w_1] \oplus \bF_2[w_1]\{h_1 v',
h_1^2 v'\} \bigr)$.  From the $A$-module presentation
of~$K$ it follows that $\Ext_A^{s,t}(K, \bF_2) = 0$ for $t-s < 8$ and for
$(s,t) = (1,11)$.  By exactness,
$\ldY{h_2}$ and $\ldY{c_0}$ in $\Ext_A(I, \bF_2)$ must map
nontrivially to $\Sigma^4 1$ and $\Sigma^4 h_1 v'$ in $E_2(\Sigma^4 ksp)$.
By $h_0$-linearity this implies that $h_0^i \cdot \ldY{h_2}$
maps to $\Sigma^4 h_0^i$ for all $i\ge0$.  
By $h_1$-linearity it also implies that $h_1 \ldY{c_0}$ maps to
$\Sigma^4 h_1^2 v'$.  These results propagate $w_1$-linearly, and show
that the image is at least as large as claimed.

To show that it is no larger, we use the presentation from
\cite{BR21}*{Prop.~3.45} of $\Ext_{A(2)}(\bF_2, \bF_2)$ as a direct sum
of cyclic $R_0$-modules, where $R_0 = \bF_2[g, w_1, w_2]$.  By inspection,
each $R_0[h_0]$-module generator of $\Ext_A(I, \bF_2)$, other than
$\ldY{h_2}$, $\ldY{c_0}$ and~$h_1 \ldY{c_0}$, maps to a trivial bidegree
of $E_2(\Sigma^4 ksp)$.  Hence all of these generators map to zero, and
the image from~$\Ext_A(I, \bF_2)$ does not meet
$\Sigma^4 \bF_2[h_0, w_1] \{v'\}$.

The uniqueness of the nonzero lifts $\ldYdZ{x}$ follows, as in the proof
of Lemma~\ref{lem:C-Hko-I}, from the fact that $\Ext_A^{s,t}(C,
\bF_2) = 0$ when $t-s = 8k+6$ and $s \geq 4k+3$.
From the presentation of $K$, with a generator in degree~$8$ that is
annihilated by~$Sq^1$, we see that $i_Z^*(\Sigma^4 v') = h_0 \cdot
\ldYdZ{h_2^2}$ is nonzero in $\Ext_A^{1,9}(K, \bF_2)$.
\end{proof}

\begin{lemma} \label{lem:dYdZ}
\leavevmode
\begin{enumerate}
\item
The kernel of $\delta_Y \delta_Z \: \Ext_A^{*,*}(K, \bF_2)
\to \Ext_A^{*+2,*}(C, \bF_2)$ is
$$
\ker(\delta_Y \delta_Z)
	= \bF_2[h_0, w_1] \{h_0\ldYdZ{h_2^2}\}
	\cong \Sigma^4 \bF_2[h_0, w_1] \{v'\} \,.
$$
\item
The cokernel of $\delta_Y \delta_Z \: \Ext_A^{*-2,*}(K, \bF_2)
	\to \Ext_A^{*,*}(C, \bF_2)$ is
$$
\cok(\delta_Y \delta_Z) = \bF_2[h_0, w_1] \oplus
	\bF_2[w_1]\{h_1, h_1^2, h_2, h_0 h_2, h_0^2 h_2, c_0, h_1 c_0\} \,.
$$
\end{enumerate}
\end{lemma}

\begin{proof}
In the exact kernel--cokernel sequence
\begin{multline*}
0 \to \ker(\delta_Z) \longto \ker(\delta_Y \delta_Z)
	\overset{\delta_Z}\longto \ker(\delta_Y) \\
\overset{\partial}\longto \cok(\delta_Z) \overset{\delta_Y}\longto
	\cok(\delta_Y \delta_Z) \longto \cok(\delta_Y) \to 0
\end{multline*}
of $\Ext_{A(2)}(\bF_2, \bF_2)$-modules,
the connecting homomorphism $\partial$ can be
identified with the composition
\begin{align*}
\bF_2[h_0, w_1]\{v\} \cong \ker(\delta_Y)
	&\into \Ext_A^{*,*}(I, \bF_2) \\
	&\onto \cok(\delta_Z)
\cong \Sigma^4(\bF_2[h_0, w_1] \oplus \bF_2[w_1]\{h_1 v', h_1^2 v'\}) \,,
\end{align*}
which maps $v$ via $h_0^3 \cdot \ldY{h_2}$ to $\Sigma^4 h_0^3$.  This is
injective, and therefore $\ker(\delta_Z) = \ker(\delta_Y \delta_Z)$.

Since $\delta_Y(\ldY{h_2}) = h_2$ and $\delta_Y(\ldY{c_0}) = c_0$,
the image $\delta_Y(\cok(\delta_Z))$ is the submodule
$ \bF_2[w_1]\{h_2, h_0 h_2, h_0^2 h_2, c_0, h_1 c_0\}$
in $\cok(\delta_Y \delta_Z)$.
%
\end{proof}

Combining these lemmas, $\Ext_A^{*,*}(K, \bF_2)$ is obtained in 
three steps from the algebra $\Ext_A^{*,*}(C, \bF_2)$ by
\begin{itemize}
\item
taking the kernel of the map onto
$$
\bF_2[h_0, w_1] \oplus
\bF_2[w_1]\{h_1, h_1^2, h_2, h_0 h_2, h_0^2 h_2, c_0, h_1 c_0\} \,,
$$ 
\item
shifting by $(s,t) = (-2,0)$ to lift over $\delta_Y\delta_Z$,
denoting the lift of $x$ by $\ldYdZ{x}$, and
\item
extending by $\Sigma^4 \bF_2[h_0, w_1] \{i_Z^*(v')\}$ above $\ldYdZ{h_2^2}$;
that is, $i_Z^*(\Sigma^4 v') = h_0 \ldYdZ{h_2^2}$.
\end{itemize}
We do not introduce notation for the identification
$$
\Ext_A^{*,*+1}(K, \bF_2) = \Ext_A^{*,*}(\Sigma^{-1} K, \bF_2) \,,
$$
but note that
$$
\ker(\delta_Y \delta_Z) 
= \Sigma^{-1} \bF_2[h_0, w_1]\{h_0\ldYdZ{h_2^2}\}
\cong \Sigma^3 \bF_2[h_0, w_1]\{v'\}
$$
when bigraded as in the right hand expression, and that 
$$ \ldYdZ{h_2^2} \in \Ext_A^{0,8}(K, \bF_2) \cong \Ext_A^{0,7}(\Sigma^{-1} K, \bF_2).$$


\begin{proposition}[\cite{Dav75}*{Thm.~1(i)}, \cite{AR05}*{Lem.~7.10(c)}]
	\label{prop:cohj}
The $A$-module extension
\begin{equation}
0 \to C \overset{i}\longto H^*(j) \overset{q}\longto
	\Sigma^{-1} K \to 0
	\tag{$E_j$} 
\end{equation}
is given by
$$
H^*(j) = \frac{A\{g_0, g_7\}}
	{A(Sq^1 g_0, Sq^2 g_0, Sq^4 g_0, Sq^8 g_0 + Sq^1 g_7,
	Sq^7 g_7, (Sq^4 Sq^6 + Sq^6 Sq^4) g_7)}
$$
with $i(1) = g_0$ and $q(g_7) = \Sigma^7 1$.  In particular, it is
induced up from an extension over~$A(3)$.
\end{proposition}

\begin{proof}
We reproduce Davis' argument: The extension of~$\Sigma^{-1} K$ is
determined by the values of $Sq^1 g_7 \in \{0, Sq^8 g_0\}$, $Sq^7 g_7
\in \{0, Sq^{14} g_0\}$ and $(Sq^4 Sq^6 + Sq^6 Sq^4)g_7 \in \{0\}$.
Here $Sq^1 Sq^7 = 0$ and $Sq^1 Sq^{14} g_0 = Sq^{15} g_0 \ne 0$, so $Sq^7
g_7 = 0$.  This leaves at most two possible extensions: the split one with
$Sq^1 g_7 = 0$ and a nonsplit one with $Sq^1 g_7 = Sq^8 g_0$.  If $(E_j)$
were split, the resulting Adams $E_2$-term would imply that $\pi_7(j)$
had order divisible by~$32$.  This contradicts Proposition~\ref{prop:pij},
and therefore $Sq^1 g_7 = Sq^8 g_0$, as asserted.
\end{proof}

\begin{proof}[Proof of Theorem~\ref{thm:E3j}]
The extension~$(E_j)$ induces a long exact sequence
$$
\dots \to \Ext_A^{s-1,t}(C, \bF_2)
	\overset{\delta_X}\longto \Ext_A^{s,t}(\Sigma^{-1} K, \bF_2)
	\overset{q^*}\longto E_2^{s,t}(j)
	\overset{i^*}\longto \Ext_A^{s,t}(C, \bF_2)
\to \cdots
$$
of modules over~$\Ext_{A(3)}(\bF_2, \bF_2)$.  We claim that
$$
\im(\delta_X) = \bF_2[h_0, w_1^2] \{\Sigma^{-1} h_0^5 \ldYdZ{h_2^2}\}
              \cong \Sigma^3 \bF_2[h_0, w_1^2] \{h_0^4 v'\} \,.
$$
From $q^* \delta_X = 0$ and Theorem~\ref{thm:id2q} we see that
$\im(\delta_X) \subset \ker(\delta_Y \delta_Z)$, which is $h_0$-torsion
free by Lemma~\ref{lem:dYdZ}.  Hence $\delta_X$ vanishes on the
$h_0$-power torsion in $\Ext_A^{*-1,*}(C, \bF_2)$, and we only need
to determine its values on monomials in $w_1$, $\alpha$ and $w_2$,
since the remaining algebra generators in Proposition~\ref{prop:ExtA2}
are $h_0$-power torsion classes.  Furthermore, by $w_1^2$-linearity
it suffices to determine $\delta_X(w_1^a \alpha^b w_2^c)$ for $a \in
\{0,1\}$, $b\ge0$ and $c\ge0$.

The bidegree containing $\delta_X(w_1)$ equals
$\bF_2\{\Sigma^{-1} h_0^5 \ldYdZ{h_2^2}\} = \bF_2\{i_Z^*( \Sigma^3 h_0^4 v')\}$.
If $\delta_X(w_1)$ were zero, then $q^*(\Sigma^{-1} h_0^{i+1} \ldYdZ{h_2^2})$
for $0 \le i \le 4$ would survive to $E_\infty(j)$ as nonzero classes,
so that $\pi_7(j)$ would have order divisible by~$32$.  This contradicts
Proposition~\ref{prop:pij}, and proves that $\delta_X(w_1) = \Sigma^{-1}
h_0^5 \ldYdZ{h_2^2}$.
(Alternatively, we can calculate enough of
$\Ext_A(H^*(j), \bF_2)$, to see that the $h_0$-tower starting in
$\Ext_A^{0,7}(H^*(j), \bF_2)$ has height exactly~$5$.)
On the other hand, $\delta_X$ maps each of the remaining
monomials $w_1^a \alpha^b w_2^c$ with $a \in \{0, 1\}$ to bidegrees
where~$\ker(\delta_Y \delta_Z)$ is trivial.  Hence the image of $\delta_X$
is as claimed.

This shows that the submodule
$$
\im(i^*) \into \Ext_A^{*,*}(C, \bF_2)
$$
is given by removing $\bF_2[h_0, w_1^2]\{w_1\}$.
It also shows that the quotient module
$$
\Ext_A^{*,*}(\Sigma^{-1} K, \bF_2) \onto \coim(q^*)
$$
is given by truncating the $h_0$-tower containing $\Sigma^{-1}  h_0
\ldYdZ{h_2^2} = i_Z^*(\Sigma^3 v')$ by setting $h_0^5 \cdot \Sigma^{-1}
\ldYdZ{h_2^2}$ equal to zero, repeated $w_1^2$-periodically.
We get a short exact sequence
$$
0 \to \coim(q^*) \longto E_2(j) \longto \im(i^*) \to 0
$$
where the projection $\coim(q^*) \onto \coim(\delta_Y \delta_Z)$ has
kernel equal to the quotient
$$
\Sigma^{-1} \bigl( \bF_2[w_1^2]\{h_0\ldYdZ{h_2^2},\, h_0^2\ldYdZ{h_2^2},\,
h_0^3\ldYdZ{h_2^2},\, h_0^4\ldYdZ{h_2^2}\}
\oplus \bF_2[h_0, w_1^2]\{w_1\cdot h_0\ldYdZ{h_2^2}\} \bigr)
$$
of $\ker(\delta_Y \delta_Z)$, and the inclusion $\im(\delta_Y \delta_Z)
\into \im(i^*)$ has cokernel
$$
\bF_2[h_0, w_1^2] \oplus
\bF_2[w_1] \{h_1, h_1^2, h_2, h_0 h_2, h_0^2 h_2, c_0, h_1 c_0\}
$$
contained in $\cok(\delta_Y \delta_Z)$.  
Additively the structure of $E_2(j)$ is the same as that given
in \cite{Dav75}*{Thm.~1(ii)}.  From the relation $Sq^8 g_0 + Sq^1 g_7$
in Proposition~\ref{prop:cohj}, it follows that $h_0\cdot q^*(\Sigma^{-1}\ldYdZ{h_2^2})
= h_3 \in E_2^{1,8}(j)$, and we adopt this notation now.

By Theorem~\ref{thm:id2q}, we know that $d_2^{s,t} \: E_2^{s,t}(j) \to
E_2^{s+2,t+1}(j)$ has rank no less than the dimension of $\coim(\delta_Y
\delta_Z)$ in bidegree~$(s,t)$, and $d_2^{s-2,t-1} \: E_2^{s-2,t-1}(j)
\to E_2^{s,t}(j)$ has rank no less than the dimension of $ \im(\delta_Y
\delta_Z)$ in the same bidegree.  Hence the dimension of $E_3^{*,*}(j)$
in each bidegree is bounded above by the corresponding dimension for
\begin{align*}
\bF_2[h_0, w_1^2]\{1, h_3 w_1\}
&\oplus \bF_2[w_1]\{h_1, h_1^2, h_2, h_0 h_2, h_0^2 h_2, c_0, h_1 c_0\} \\
&\oplus \bF_2[w_1^2]\{h_3, h_0 h_3, h_0^2 h_3, h_0^3 h_3\} \,.
\end{align*}
See Figure~\ref{upper} for a picture of this upper bound.

Since we know that $\pi_{8k-1}(j) \cong
\bZ/16 k$ for $k\ge1$ (modulo odd torsion) there must also be
$d_2$-differentials
$$
d_2(w_1^k) = h_0^5 h_3 w_1^{k-1}
$$
for positive $k \equiv 2 \mod 4$, while $d_2(w_1^k) = 0$ for $k \equiv
0 \mod 4$.  Hence the upper bound for the dimension of $E_3(j)$ is
not exactly attained.  In particular, the rank of $d_2$ is one larger
than the rank of $\delta_Y \delta_Z$ in some bidegrees
for $t-s \equiv 16 \mod 32$.

The remaining differential pattern also follows from the known order of
these homotopy groups, since the length of the differential on
$w_1^k$ determines the order of $\pi_{8k-1}(j)$, and vice versa.
\end{proof}

\begin{remark}
In view of the extensions $h_0 \cdot q^*(\Sigma^{-1} \ldYdZ{h_2^2}) = h_3$
and $h_0^2 \cdot d_0 = h_2^2 w_1$, it follows that
$E_2(j)$ contains an $h_0$-tower of height~$5$ generated by 
$q^*(\Sigma^{-1} \ldYdZ{h_2^2})$ and $h_0$-towers of height~$7$ generated
by $q^*(\Sigma^{-1}\ldYdZ{d_0 w_1^a})$ for odd $a\ge1$, together with
infinite $h_0$-towers on $q^*(\Sigma^{-1}\ldYdZ{d_0 w_1^a})$ for even $a\ge0$.  
These extra classes near the bottom of the $h_0$-towers support
$d_2$-differentials, hence are no longer present at the $E_3$-term.
\end{remark}


\begin{remark}
The map $e \: E_2(S) \to E_2(j)$ of mod~$2$ Adams spectral sequences
detects some $d_2$-differentials in the domain, including $d_2(f_0) =
h_0^2 e_0$ and $d_2(h_0 i) = h_0^2 Pd_0$.  These differentials easily
propagate to determine those on~$e_0$, $i$, $j$, $k$, $\ell$, $m$ and~$y$,
in the usual notation~\cite{Tan70}, \cite{BR21}*{Ch.~11}.
\end{remark}

\section{The mod~$2$ image-of-$J$ spectrum}
\label{sec:jmod2}

An analysis similar to, but much easier than, the preceding leads
to the conclusion that the Adams spectral sequence for the mod~$2$
reduction~$j/2$ of the connective image-of-$J$ spectrum has $E_3$-term
as shown in Figure~\ref{fig:jmod2}.  It is then immediate that $E_3(j/2)
= E_\infty(j/2)$.  A summary of the argument follows.

The mod~$2$ Moore spectrum $S/2 = S \cup_2 e^1$ has cohomology~$H^*(S/2)
= E[Sq^1]$, where $E[-]$ denotes the exterior algebra.
Smashing~\eqref{eq:j} with this Moore spectrum, we obtain a homotopy
(co-)fiber sequence
\begin{equation} \label{eq:jmod2}
j/2 \longto ko/2 \overset{\psi/2}\longto \Sigma^4 ksp/2 \longto \Sigma j/2
\,,
\end{equation}
which we view as a case of the homotopy cofiber sequence~\eqref{eq:XYZ}.

The mod~$2$ cohomology $A$-modules and associated $\Ext$ groups for
the $ko$-module spectra~$ko/2$ and~$ksp/2$ are easily calculated from
Proposition~\ref{prop:cohko}.  As is well known, $\Ext_A(H^*(ko/2),
\bF_2)$ consists of a ``lightning flash'' repeated $w_1$-periodically, while
$\Ext_A^{s,t}(H^*(ksp/2), \bF_2) \cong \Ext_A^{s+3,t+7}(H^*(ko/2),\bF_2)$.
We make these explicit in Proposition~\ref{prop:ExtA2M1}.
Similarly, the kernel $K_2 = K \otimes E[Sq^1]$, image $I_2 = I \otimes
E[Sq^1]$ and cokernel $C_2 = C \otimes E[Sq^1]$ of $(\psi/2)^*$ can
readily be presented as $A$-modules.  However, we only need their
connectivities, which are immediate from Proposition~\ref{prop:cohko}.

As in the integral case, $\Ext_A(K_2, \bF_2)$, $\Ext_A(I_2, \bF_2)$
and $\Ext_A(C_2, \bF_2)$ are closely related, where the latter is
explicitly known.

\begin{proposition}[\cite{BR21}*{Prop.~4.2}] \label{prop:ExtA2M1}
As an $\bF_2[h_0, h_1, w_1]$-module
$$
\Ext_A(C_2, \bF_2) = \Ext_{A(2)}(E[Sq^1], \bF_2)
$$
contains 
\begin{enumerate}
\item
\label{case1}
a summand isomorphic to $\Ext_A(H^*(ko/2), \bF_2)$
that is free over $\bF_2[w_1]$ on
classes forming a ``lightning flash''
$$
\begin{tabular}{ >{$}c<{$}|>{$}c<{$} >{$}c<{$} >{$}c<{$} >{$}c<{$}
         >{$}c<{$} >{$}c<{$} >{$}c<{$} }
& i(1) & i(h_1) & i(h_1^2) &
    \widetilde{h_1} & h_1\widetilde{h_1} & h_1^2\widetilde{h_1} \\
\hline
s & 0 & 1 & 2 & 1 & 2 & 3 \\
t & 0 & 2 & 4 & 3 & 5 & 7 \\
t-s & 0 & 1 & 2 & 2 & 3 & 4
\end{tabular}
$$
and
\item
\label{case2}
a summand
isomorphic to $\Sigma^{1,0} \Ext_A(H^*(\Sigma^4 ksp/2), \bF_2)$
that is free over $\bF_2[w_1]$ on
classes forming a ``shifted lightning flash''
$$
\begin{tabular}{ >{$}c<{$}|>{$}c<{$} >{$}c<{$} >{$}c<{$} >{$}c<{$}
         >{$}c<{$} >{$}c<{$} >{$}c<{$} }
& i(h_2) & \widetilde{h_2^2} & h_1 \widetilde{h_2^2} & h_1^2 \widetilde{h_2^2} &
    \widetilde{c_0} & h_1\widetilde{c_0}  \\
\hline
s & 1 & 2 & 3 & 4 & 3 & 4 \\
t & 4 & 9 & 11 & 13 & 12 & 14 \\
t-s & 3 & 7 & 8 & 9 & 9 & 10
\end{tabular}
$$
\end{enumerate}
together with other direct summands.
\end{proposition}

\begin{remark}
In the shifted lightning flash, we have $w_1 \cdot i(h_2) = h_1^2 \widetilde{c_0}$.
\end{remark}

\begin{remark}
The notation in Proposition~\ref{prop:ExtA2M1} is that
of~\cite{BR21}*{\S4.1}:  $i(x)$ denotes the image of a class $x$ from the
bottom cell of $S \cup_2 e^1$, while $\widetilde{x}$ projects to a class
$x$ on the top cell.
\end{remark}

\begin{lemma} \label{lem:dYdZmod2}
The composite 
$$
\delta_Y \delta_Z \: \Ext_A^{*,*}(K_2, \bF_2)
	\longto \Ext_A^{*+2,*}(C_2, \bF_2)
$$
is a monomorphism, with cokernel
$$
  \bF_2[w_1]\{ i(1), i(h_1), i(h_1^2),
    \widetilde{h_1}, h_1\widetilde{h_1}, h_1^2\widetilde{h_1},
   i(h_2), \widetilde{h_2^2}, h_1 \widetilde{h_2^2}, h_1^2 \widetilde{h_2^2},
    \widetilde{c_0}, h_1\widetilde{c_0}\} \,.
$$
\end{lemma}

\begin{proof}
By Proposition~\ref{prop:KICStructure} the module $I_2$ is
$3$-connected and the module $K_2$ is $7$-connected.  Therefore, $q_Y^*
: \Ext_A^{s,t}(C_2, \bF_2) \to E_2^{s,t}(ko/2)$ is an isomorphism for
$t-s \leq 2$, so that $q_Y^*(i(1))$ is nonzero.  Then, by $h_0$-, $h_1$-
and $w_1$-linearity, $q_Y^*$ is an epimorphism, and hence $\delta_Y$
is a monomorphism with cokernel the module in case~(\ref{case1}) of
Proposition~\ref{prop:ExtA2M1}.  Similarly, $q_Z^* : \Ext_A^{s,t}(I_2,
\bF_2) \to E_2^{s,t}(\Sigma^4 ksp/2)$ is an isomorphism for $t-s
\leq 6$.  Again by $h_0$-, $h_1$- and $w_1$-linearity, $q_Z^*$ is
an epimorphism, and hence $\delta_Z$ is a monomorphism with cokernel
mapping isomorphically by~$\delta_Y$ to the module in case~(\ref{case2})
of Proposition~\ref{prop:ExtA2M1}.
\end{proof}

\begin{theorem} \label{thm:E3jmod2}
Let $j/2 = j \wedge S/2$ be the mod~$2$ connective image-of-$J$ spectrum.
There is an isomorphism
\begin{align*}
  E_3(j/2) &= E_\infty(j/2) \\
	&\cong \bF_2[w_1]\{ i(1), i(h_1), i(h_1^2),
    \widetilde{h_1}, h_1\widetilde{h_1}, h_1^2\widetilde{h_1},
   i(h_2), \widetilde{h_2^2}, h_1 \widetilde{h_2^2}, h_1^2 \widetilde{h_2^2},
    \widetilde{c_0}, h_1\widetilde{c_0}\} \,.
\end{align*}
The bidegrees of these generators are as in Proposition~\ref{prop:ExtA2M1}.
\end{theorem}

\begin{proof}
By Theorem~\ref{thm:id2q}, $i^*d_2q^* = \delta_Y\delta_Z$ and by 
Lemma~\ref{lem:dYdZmod2} this is a monomorphism.  Therefore, $q^*$ is 
a monomorphism and the sequence
$$
0 \to \Ext_A^{s,t}(\Sigma^{-1} K_2, \bF_2)
	\overset{q^*}\longto E_2^{s,t}(j/2)
	\overset{i^*}\longto \Ext_A^{s,t}(C_2, \bF_2)
\to 0
$$
is short exact.  It follows from Theorem~\ref{thm:id2q} that
$\cok(\delta_Y\delta_Z)$ is an upper bound for $E_3(j/2)$. It then
follows that $E_3(j/2) \cong \cok(\delta_Y\delta_Z)$ for bidegree
reasons.  Similarly, there is no room for further differentials.
\end{proof}

\begin{remark}
The cyclic group $\pi_{8k-1}(j)$ must map onto $\pi_{8k-1}(j/2)
\cong \bZ/2$ with a filtration shift of~$1$.  The nontrivial
$\eta$-multiplication on this latter group then implies the hidden
$\eta$-extension from $h_3 w_1^{k-1}$ to $c_0 w_1^{k-1}$ for~$j$, as
shown in Figure~\ref{infty}.
\end{remark}

\section{The odd-primary image-of-$J$ spectrum}
\label{sec:jpodd}

Let $p$ be an odd prime, and let all spectra and homotopy groups be
implicitly $p$-completed.  The $p$-primary connective image-of-$J$
spectrum~$j$ sits in a homotopy (co-)fiber sequence
\begin{equation} \label{eq:jp}
j \longto \ell \overset{\psi}\longto \Sigma^q \ell
	\longto \Sigma j \,,
\end{equation}
where $q = 2p-2$, $r$ generates $\bZ_p^\times$ topologically, $\ell$ is
the Adams summand of connective complex $K$-theory, and $\psi$ denotes
a lift of $\psi^r - 1 \: \ell \to \ell$.  Let $E(1) = E[\beta, Q_1]$ and
$A(1) = \< \beta, P^1 \>$ be the sub Hopf algebras of the mod~$p$ Steenrod
algebra~$A$ generated by the listed elements, where $Q_1 = P^1 \beta -
\beta P^1$.  Let $a \doteq b$ mean that $a$ is a $p$-adic unit times~$b$.

\begin{proposition}[\cite{Sin68}]
\leavevmode
\begin{enumerate}
\item
$H^*(\ell) = A/A(\beta, Q_1) = A /\!/ E(1)$ and
$E_2(\ell) = \Ext_{E(1)}(\bF_p, \bF_p) = \bF_p[v_0, v_1]$
with algebra generators in $(s,t)$-bidegrees~$|v_0| = (1,1)$
and~$|v_1| = (1,q+1)$.
\item
$\pi_*(\ell) = \bZ[v_1]$ with $|v_1| = q$.
\end{enumerate}
\end{proposition}

\begin{proposition}[\cite{Ada62}]
\leavevmode
\begin{enumerate}
\item
$\psi^r(v_1) = r^{p-1} v_1$.
\item
$\pi_0(j) = \bZ$ and, when $n\ge1$,
$\pi_n(j) = \begin{cases}
\bZ/pk & \text{for $n = kq-1$,} \\
0 & \text{otherwise.}
\end{cases}$
\end{enumerate}
\end{proposition}

\begin{proposition}[\cite{Rog03}*{Prop.~5.1(b)}]
\label{prop:KICp3}
\leavevmode
\begin{enumerate}
\item
$\psi^* \: H^*(\Sigma^q \ell) \to H^*(\ell)$ is determined by
$\psi^*(\Sigma^q 1) \doteq P^1$.
\item
$K = \ker(\psi^*) = \Sigma^{pq} A/\!/A(1)$.
\item
$I = \im(\psi^*) = \Sigma^q A/A(\beta, Q_1, (P^1)^{p-1})$.
\item
$C = \cok(\psi^*) = A/\!/A(1)$.
\end{enumerate}
In particular, $K \cong \Sigma^{pq} C$ as $A$-modules.
\end{proposition}

\begin{proposition}
\begin{align*}
\Ext_A(C, \bF_p) &= \Ext_{A(1)}(\bF_p, \bF_p) \\
	&= \bF_p[v_0, b, w_1]
		\otimes E[a_i \mid 0<i<p]/(\sim)
\end{align*}
is a free $\bF_p[w_1]$-module, where the ideal
$(\sim)$ imposes the relations $v_0 a_i = 0$ and
$$
a_i a_j = \begin{cases}
	(-1)^{i-1} v_0^{p-2} b & \text{for $i+j=p$,} \\
	0 & \text{otherwise.}
\end{cases}
$$
The generators are graded as follows.
$$
\begin{tabular}{ >{$}c<{$}|>{$}c<{$} >{$}c<{$} >{$}c<{$} >{$}c<{$} }
& v_0 & a_i & b & w_1 \\
\hline
s & 1 & i & 2 & p \\
t & 1 & i(q+1)-1 & pq & p(q+1) \\
t-s & 0 & iq-1 & pq-2 & pq
\end{tabular}
$$
\end{proposition}

\begin{proof}
This is asserted without signs in~\cite{Hil08}*{Thm.~3.6}.  We give
a proof using the multiplicative Davis--Mahowald spectral sequence
of~\cite{BR21}*{Ch.~2}.  Let $P(0) = \< P^1 \> = \bF_p[P^1]/((P^1)^p)$ be
the sub Hopf algebra of $A(1)$ generated by $P^1$, so that $A(1)/\!/P(0)
= E[\beta, Q_1]$ as a left $A(1)$-module quotient coalgebra.  Dually,
$P(0)_* = \bF_p[\xi_1]/(\xi_1^p)$ is a quotient Hopf algebra of $A(1)_* =
E[\tau_0, \tau_1] \otimes \bF_p[\xi_1]/(\xi_1^p)$, with $(A(1)/\!/P(0))_*
= E[\tau_0, \tau_1]$ as a left $A(1)_*$-comodule subalgebra.  Let $R^*
= \bF_p[v_0, v_1]$ be the graded $A(1)_*$-comodule algebra with coaction
$\nu(v_0) = 1 \otimes v_0$ and $\nu(v_1) = 1 \otimes v_1 + \xi_1 \otimes
v_0$, and give
$$
(A(1)/\!/P(0))_* \otimes R^*
$$
the differential $\delta(\tau_0) = v_0$, $\delta(\tau_1) = v_1$,
$\delta(v_0) = 0$, $\delta(v_1) = 0$ and the diagonal $A(1)_*$-coaction,
making it a differential graded $A(1)_*$-comodule algebra resolution
of~$\bF_p$.
We obtain an algebra spectral sequence
$$
E_1^{\sigma,s,t} = \Ext_{P(0)_*}^{s-\sigma,t}(\bF_p, R^\sigma)
	\Longrightarrow_\sigma \Ext_{A(1)_*}^{s,t}(\bF_p, \bF_p)
		\cong \Ext_{A(1)}^{s,t}(\bF_p, \bF_p) \,.
$$
Since $v_1^p \in R^p$ is $A(1)_*$-comodule primitive, there is an
extension of algebra spectral sequences $\bF_p[v_1^p] \to E_1^{*,*,*}
\to \bar E_1^{*,*,*}$ with
$$
\bar E_1^{\sigma,s,t} = \Ext_{P(0)_*}^{s-\sigma,t}(\bF_p, \bar R^\sigma) \,,
$$
where $\bar R^* = R^*/(v_1^p)$.  Here $\bar R^{i-1} = \bF_p\{v_0^{i-1},
\dots, v_1^{i-1}\}$ for $0<i<p$, while $\bar R^{i-1}
= \bF_p\{v_0^{i-1}, \dots, v_0^{i-p} v_1^{p-1}\} \cong P(0)_*
\{v_0^{i-p}\}$ for $i \ge p$.
For $0<i<p$ we have minimal injective resolutions of period~$2$
$$
0 \to \bar R^{i-1} \overset{\eta}\longto
	P(0)_* \{v_0^{i-1}\} \overset{\delta^0}\longto
	\Sigma^{iq} P(0)_* \{v_0^{i-1}\} \overset{\delta^1}\longto
	\Sigma^{pq} P(0)_* \{v_0^{i-1}\} \longto \dots
$$
with $\delta^0$ dual to multiplication by $(P^1)^i$ and $\delta^1$
dual to multiplication by $(P^1)^{p-i}$, which implies that
$$
\Ext_{P(0)_*}^{*,*}(\bF_p, \bar R^{i-1})
	= E[a'_i] \otimes \bF_p[b] \{v_0^{i-1}\}
$$
where $|a'_i| = (1, iq)$ and $|b| = (2, pq)$.  For $i\ge p$, on the
other hand, we have $\Ext_{P(0)_*}^{*,*}(\bF_p, \bar R^{i-1}) =
\bF_p\{v_0^{i-1}\}$.
The product $\phi \: \bar R^{i-1} \otimes \bar R^{p-i-1}
\to \bar R^{p-2}$ extends to a chain map $\phi_*$ from the tensor product
of the injective resolutions for $\bar R^{i-1}$ and $\bar R^{p-i-1}$
to the injective resolution for $\bar R^{p-2}$, in which
$$
\phi_2 \: \Sigma^{iq} P(0)_* \{v_0^{i-1}\}
\otimes \Sigma^{(p-i)q} P(0)_* \{v_0^{p-i-1}\}
\longto \Sigma^{pq} P(0)_* \{v_0^{p-2}\}
$$
maps $v_0^{i-1} \otimes v_0^{p-i-1}$ to $(-1)^{i-1} v_0^{p-2}$.
(The reader may prefer to verify this in the dual context of projective
$P(0)$-module resolutions, keeping in mind that $\binom{p-1}{i-1}
\equiv (-1)^{i-1} \mod p$.)  It follows that
$$
v_0^{i-1} a'_i \cdot v_0^{p-i-1} a'_{p-i} = (-1)^{i-1} v_0^{p-2} b
$$
in $E_1^{*,*,*}$.  On the other hand, the products $v_0 \cdot v_0^{i-1}
a'_i$ and $v_0^{i-1} a'_i \cdot v_0^{j-1} a'_{j-i}$ for $i+j \ne p$
are zero because they lie in trivial groups.  The Davis--Mahowald
spectral sequence collapses at the $E_1$-term, since all differentials
on the algebra generators $v_0$, $v_0^{i-1} a'_i$, $b$ and $v_1^p$
land in trivial groups.  Furthermore, there is no room
for hidden multiplicative extensions.  Letting
$a_i \in \Ext_{A(1)}^{i-1+1,i-1+iq}(\bF_p, \bF_p)$,
$b \in \Ext_{A(1)}^{2,pq}(\bF_p, \bF_p)$ and $w_1 \in
\Ext_{A(1)}^{p,p(q+1)}(\bF_p, \bF_p)$ be detected by $v_0^{i-1} a'_i$,
$b$ and $v_1^p$, respectively, we obtain the stated computation of
$\Ext_{A(1)}(\bF_p, \bF_p)$.
\end{proof}

\begin{lemma}
The $A$-module extension
\begin{align*}
0 \from C \overset{q_Y}\longfrom H^*(\ell)
        \overset{i_Y}\longfrom I \from 0
        \tag{$E_\ell$}
\end{align*}
induces a long exact sequence
$$
\dots \to E_2^{s-1,t}(\ell)
        \overset{i_Y^*}\longto \Ext^{s-1,t}_A(I, \bF_p)
        \overset{\delta_Y}\longto \Ext^{s,t}_A(C, \bF_p)
        \overset{q_Y^*}\longto E_2^{s,t}(\ell)
\to \cdots
$$
with
$$
\bF_p[v_0, w_1]\{v_1, \dots, v_1^{p-1}\}
	\cong \ker(\delta_Y) \into \Ext_A^{*,*}(I, \bF_p)
$$
and
$$
\Ext_A^{*,*}(C, \bF_p) \onto \cok(\delta_Y) \cong \bF_p[v_0, w_1]
\,.
$$
For each nonzero $x \in \Ext_A^{s,t}(C, \bF_p)$ that maps trivially to
$\cok(\delta_Y)$ there is a unique lift $\ldY{x} \in \Ext_A^{s-1,t}(I,
\bF_p)$ satisfying $\delta_Y(\ldY{x}) = x$, and $\Ext_A(I, \bF_p)$
consists of these $\ldY{x}$, extended by the free $\bF_p[v_0,
w_1]$-module on $i_Y^*(v_1^i) = v_0 \ldY{a_i}$ for $0<i<p$.
\end{lemma}

\begin{lemma}
The $A$-module extension
\begin{align*}
0 \from I \overset{q_Z}\longfrom H^*(\Sigma^q \ell)
	\overset{i_Z}\longfrom K \from 0
	\tag{$E_{\Sigma^q \ell}$}
\end{align*}
induces a long exact sequence
$$
\dots \to E_2^{s-1,t}(\Sigma^q \ell)
	\overset{i_Z^*}\longto \Ext^{s-1,t}_A(K, \bF_p)
	\overset{\delta_Z}\longto \Ext^{s,t}_A(I, \bF_p)
	\overset{q_Z^*}\longto E_2^{s,t}(\Sigma^q \ell)
\to \cdots
$$
with
$$
\Sigma^q \bF_p[v_0, w_1] \{v_1^{p-1}\}
	\cong \ker(\delta_Z) \into \Ext_A^{*,*}(K, \bF_p)
$$
and
$$
\Ext_A^{*,*}(I, \bF_p) \onto \cok(\delta_Z) \cong
    \Sigma^q \bF_p[v_0, w_1] \{1, \dots, v_1^{p-2}\} \,.
$$
For each nonzero $\ldY{x} \in \Ext_A^{s,t}(I, \bF_p)$ that maps trivially to
$\cok(\delta_Z)$ there is a unique lift $\ldYdZ{x} \in \Ext_A^{s-1,t}(K, \bF_p)$
satisfying $\delta_Z(\ldYdZ{x}) = \ldY{x}$, and $\Ext_A(K, \bF_p)$ consists
of these $\ldYdZ{x}$, extended by the free $\bF_p[v_0, w_1]$-module
on $i_Z^*(\Sigma^q v_1^{p-1}) = v_0^{p-1} \ldYdZ{b}$.
\end{lemma}

\begin{remark}
$\Ext_A(K, \bF_p)$ is a free 
$\Ext_{A(1)}(\bF_p, \bF_p)$-module on one generator, $\ldYdZ{b}$,
as Proposition~\ref{prop:KICp3} indicates.
\end{remark}

\begin{lemma}
\leavevmode
\begin{enumerate}
\item
The kernel of $\delta_Y \delta_Z \: \Ext_A^{*,*}(K, \bF_p)
\to \Ext_A^{*+2,*}(C, \bF_p)$ is
$$
\ker(\delta_Y \delta_Z)
	= \bF_p[v_0, w_1] \{v_0^{p-1} \ldYdZ{b}\} \,.
$$
\item
The cokernel of $\delta_Y \delta_Z \: \Ext_A^{*-2,*}(K, \bF_p)
	\to \Ext_A^{*,*}(C, \bF_p)$ is
$$
\cok(\delta_Y \delta_Z) = \bF_p[v_0, w_1] \oplus
	\bF_p[w_1]\{a_1, \dots, a_{p-1}\} \,.
$$
\end{enumerate}
\end{lemma}

\begin{proposition}[\cite{Rog03}*{Prop.~5.1(b)}]
The $A$-module extension
\begin{equation*}
0 \to C \overset{i}\longto H^*(j) \overset{q}\longto \Sigma^{-1} K \to 0
	\tag{$E_j$}
\end{equation*}
is given by
$$
H^*(j) = \frac{A\{g_0, g_{pq-1}\}}
{A(\beta(g_0), P^1(g_0), P^p(g_0) \doteq \beta(g_{pq-1}), P^1(g_{pq-1}))}
$$
with $i(1) = g_0$ and $q(g_{pq-1}) = \Sigma^{pq-1} 1$.  In particular,
it is induced up from an extension over $A(2) = \< \beta, P^1, P^p \>$.
\end{proposition}

\begin{proof}
By change-of-rings, $\Ext_A^1(\Sigma^{-1} K, C) \cong
\Ext_{A(1)}^1(\Sigma^{pq-1} \bF_p, A/\!/A(1)) \cong \bF_p$.  If the
extension were trivial, then
$$
E_2(j) \cong \Ext_{A(1)}(\bF_p, \bF_p)
	\oplus \Ext_{A(1)}(\Sigma^{pq-1} \bF_p, \bF_p)
$$
and the three classes $\Sigma^{-1} v_0^i \ldYdZ{b}$ for $p-1 \le i
\le p+1$ would survive to $E_\infty(j)$.  This contradicts $\pi_{pq-1}(j)
\cong \bZ/p^2$.  Hence the extension is nontrivial, which can only happen
if $\beta(g_{pq-1})$ is a unit times $P^p(g_0)$.
\end{proof}

\begin{lemma}
The $A$-module extension~$(E_j)$ induces a long exact sequence
$$
\dots \to \Ext_A^{s-1,t}(C, \bF_p)
	\overset{\delta_X}\longto \Ext_A^{s,t}(\Sigma^{-1} K, \bF_p)
	\overset{q^*}\longto E_2^{s,t}(j)
	\overset{i^*}\longto \Ext_A^{s,t}(C, \bF_p) \to \dots
$$
where $\delta_X$ is a derivation that vanishes on $v_0$, $a_i$
and $b$, with $\delta_X(w_1) \doteq \Sigma^{-1} v_0^{p+1}
\ldYdZ{b}$.
Hence $\delta_X$ maps
$$
\coim(\delta_X) = \bF_p[v_0]\{w_1^k
	\mid k \not\equiv 0 \mod p\}
$$
isomorphically to
$$
\im(\delta_X) = \bF_p[v_0]\{\Sigma^{-1} v_0^{p+1} w_1^{k-1} \ldYdZ{b}
	\mid k \not\equiv 0 \mod p\} \,.
$$
\end{lemma}

\begin{definition}
Let $a_p \in E_2(j)$ be the image of $a_0^{p-2} h_1$ in $E_2(S)$.
\end{definition}

Up to a unit in $\bF_p$, this is also the image $q^*(\Sigma^{-1} v_0^{p-1}
\ldYdZ{b}) = q^*(i_Z^*(\Sigma^{q-1} v_1^{p-1}))$ of classes $\Sigma^{-1}
v_0^{p-1} \ldYdZ{b} \in \Ext_A(\Sigma^{-1} K, \bF_p)$ and $\Sigma^{q-1}
v_1^{p-1} \in E_2(\Sigma^{q-1} \ell)$.
At these primes, $E_2(j)$ has Krull dimension~$3$, while $E_3(j)$ has
Krull dimension~$2$.

\begin{theorem} \label{thm:E3jpodd}
Let $j$ be the connective image-of-$J$ spectrum at an odd prime~$p$.
There is an isomorphism
\begin{align*}
E_3(j) &\cong \bF_p[w_1] \{a_1, \dots, a_{p-1}\} \\
	&\quad\oplus \bF_p[w_1^p]
		\{v_0^i a_p w_1^{k-1}
		\mid \text{$0 \le i \le 1$, $0<k<p$} \} \\
	&\quad\oplus \bF_p[w_1^{p^2}]
		\{v_0^i a_p w_1^{pk-1}
		\mid \text{$0 \le i \le 2$, $0<k<p$} \} \\
	&\quad\oplus \bF_p[v_0, w_1^{p^2}]
		\{1, a_p w_1^{p^2-1} \}
\end{align*}
with generators in $(s,t)$-bidegrees $|v_0| = (1,1)$, $|a_i| =
(i,i(q+1)-1)$ for $0<i<p$, $|a_p| = (p-1,p(q+1)-2)$ and~$|w_1| =
(p,p(q+1))$.  The remaining nonzero differentials are
$$
d_r(v_0^i w_1^k) \doteq v_0^{r+1} a_p w_1^{k-1}
$$
for $r\ge3$, $i\ge0$ and $\ord_p(k) = r-1$.
Hence
\begin{align*}
E_\infty(j) &\cong \bF_p[v_0] \\
	&\quad\oplus \bF_p[w_1]\{a_1, \dots, a_{p-1}\} \\
	&\quad\oplus \bigoplus_{r\ge1}
		\bF_p[w_1^{p^r}]\{v_0^i a_p w_1^{p^{r-1}k-1}
		\mid \text{$0 \le i \le r$, $0<k<p$} \} \,.
\end{align*}
\end{theorem}

\begin{proof}
By Theorem~\ref{thm:id2q}, the differential $d_2 \: E_2(j) \to E_2(j)$
maps the images under~$q^*$ of the classes $\Sigma^{-1} \ldYdZ{x}$ to
lifts over~$i^*$ of the classes~$x$, mapping only the classes
$$
\bF_p\{\Sigma^{-1} v_0^{p-1+i} w_1^{k-1} \ldYdZ{b}
	\mid \text{$0\le i\le1$ if $k \not\equiv 0 \mod p$} \}
$$
of $\ker(\delta_Y\delta_Z)/\im(\delta_X)$ to zero.
Lifts over~$i^*$ of most classes~$x$ become $d_2$-boundaries
this way, leaving only the classes
$$
\bF_p[v_0, w_1^p] \oplus \bF_p[w_1]\{a_1, \dots, a_{p-1}\}
$$
of $\cok(\delta_Y\delta_Z) \cap \ker(\delta_X)$.
The direct sum of these two bigraded groups gives an upper bound for
$E_3(j)$, and comparing this with the known abutment, we must also have
nonzero differentials
$$
d_2(w_1^{pk}) \doteq \Sigma^{-1} v_0^{p+2} w_1^{pk-1} \ldYdZ{b}
$$
for all $k \not\equiv 0 \mod p$.  Substituting $a_p$ for the image
under~$q^*$ of $\Sigma^{-1} v_0^{p-1} \ldYdZ{b}$, we obtain the
stated Adams $E_3$-term for~$j$.  The known abutment also determines
the later differentials and the $E_\infty$-term.
\end{proof}

\begin{remark}
See Figure~\ref{fig:E2d2jp3} for the case~$p=3$.
The map $e \: E_2(S) \to E_2(j)$ of mod~$p$ Adams spectral sequences
detects some $d_2$-differentials in the domain, including those on
$h_1$, $g_2$ and~$u$ in the notation of~\cite{Nak75}.  On the classes
in the image of the $J$-homomorphism, the map $e \: \pi_{kq-1}(S) \to
\pi_{kq-1}(j)$ preserves Adams filtration if $\ord_p(k) \in \{0,1\}$,
and otherwise increases Adams filtration by~$\ord_p(k)-1$.
\end{remark}

\begin{theorem}
The Adams spectral sequence for $j/p = j \wedge (S \cup_p e^1)$ with $p$
odd has
$$
E_2(j) = E[a_1, \Sigma^{-1} \ldYdZ{b}] \otimes \bF_p[v_1, b]
$$
with $d_2(\Sigma^{-1} \ldYdZ{b}) = \pm b$ and
$$
E_3(j) = E_\infty(j) = E[a_1] \otimes \bF_p[v_1] \,.
$$
Here $|a_1| = (1,q)$, $|\Sigma^{-1} \ldYdZ{b}| = (0,pq-1)$, $|v_1| =
(1,q+1)$ and~$|b| = (2,pq)$.
\end{theorem}

\begin{proof}
See \cite{Rav86}*{Thm.~3.1.28} for $\Ext_{A(1)}(E[\beta], \bF_p)$.
\end{proof}


\begin{landscape}
\begin{figure}

\caption{$(E_2(j), d_2)$ at $p=2$ \label{fig:E2d2j}}
\end{figure}
\end{landscape}

\begin{landscape}
\begin{figure}
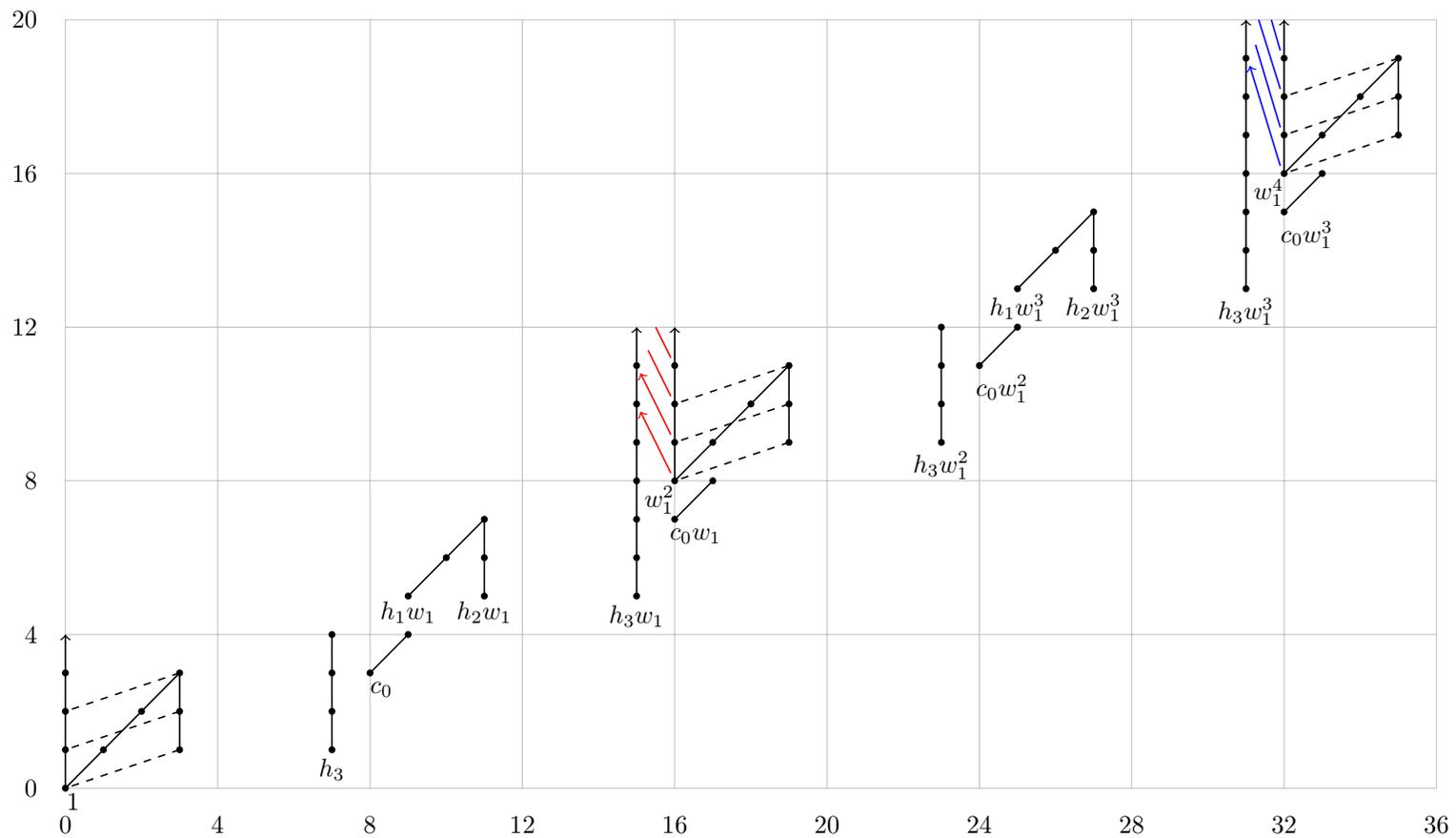

%
\caption{Upper bound for $E_3(j)$ with $d_2$ and $d_3$ \label{upper}}
\end{figure}
\end{landscape}

\begin{landscape}
\begin{figure}
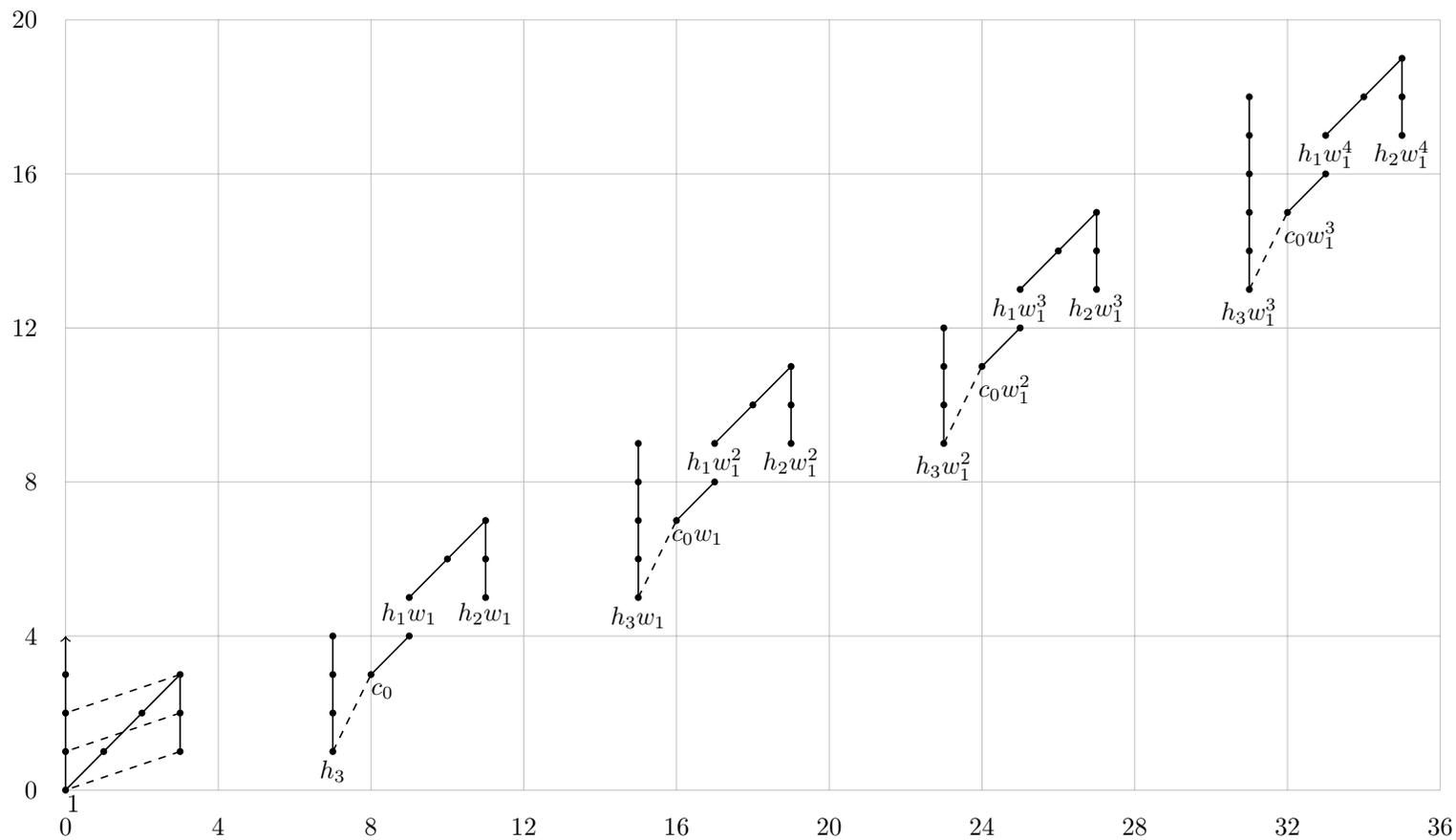

%
\caption{$E_\infty(j)$ with hidden $\eta$-extensions \label{infty}}
\end{figure}
\end{landscape}

\begin{landscape}
\begin{figure}
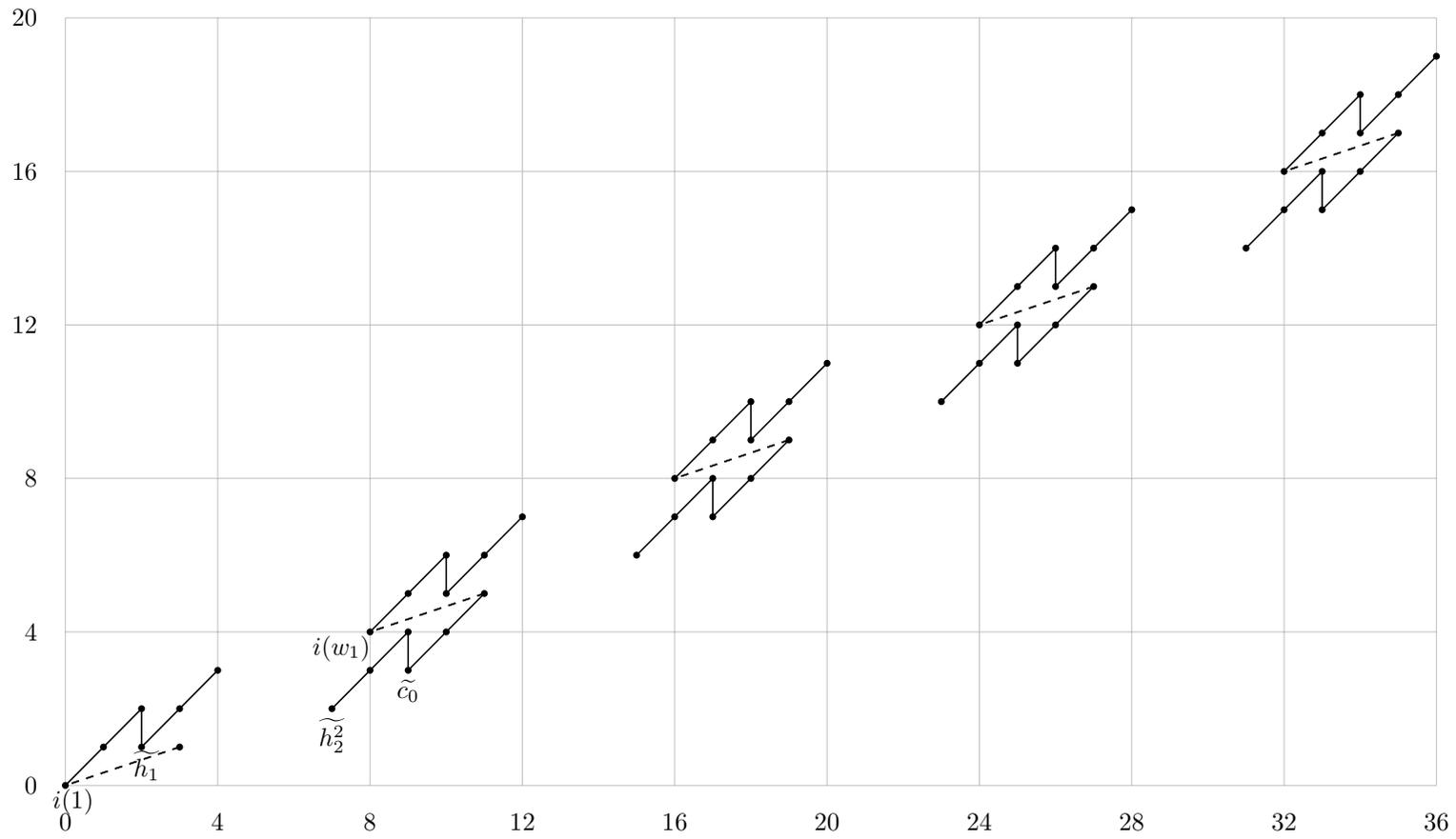

%

\caption{$E_3(j/2) = E_\infty(j/2)$ \label{fig:jmod2}}
\end{figure}
\end{landscape}

\begin{landscape}
\begin{figure}
\vspace{1.5cm} 

%
\caption{$(E_2(j), d_2)$ at $p=3$ \label{fig:E2d2jp3}}
\end{figure}
\end{landscape}

\end{document}